\documentclass[reqno, 12pt]{article}

\usepackage[utf8]{inputenc}
\usepackage[margin=3cm]{geometry}
\usepackage{amsfonts,amsopn}
\usepackage{hyperref}
\hypersetup{
colorlinks=true,
linkcolor=blue,
citecolor=cyan
}
\usepackage{amsmath,amssymb}
\usepackage{amsthm}
\usepackage{xcolor}
\usepackage{tikz}
\usepackage{comment}
\usepackage{ulem}
\usepackage{listings}

\usepackage{graphicx} 
\definecolor{mauve}{rgb}{0.58,0,0.82}
\definecolor{dkgreen}{rgb}{0,0.6,0}
\newtheorem{theorem}{Theorem}

\newtheorem{proposition}[theorem]{Proposition}
\newtheorem{corollary}[theorem]{Corollary}
\newtheorem{lemma}[theorem]{Lemma}
\newtheorem{claim}[theorem]{Claim}

\newcommand{\Ga}{{\sf bull}}
\newcommand{\Gb}{{\sf dart}}
\newcommand{\Gc}{{\sf house}}
\newcommand{\Gd}{{\sf gem}}
\newcommand{\Ge}{{\sf full-house}}
\newcommand{\Gf}{{\sf G_{6,5}}}
\newcommand{\Gg}{{\sf 5-pan}}
\newcommand{\Gh}{{\sf G_{6,7}}}
\newcommand{\Gi}{{\sf G_{6,8}}}
\newcommand{\Gj}{{\sf G_{6,9}}}
\newcommand{\Gk}{{\sf G_{6,10}}}
\newcommand{\Gl}{{\sf co-twin-house}}
\newcommand{\Gm}{{\sf G_{6,12}}}
\newcommand{\Gn}{{\sf co\text{-}twin\text{-}C_5}}
\newcommand{\Go}{{\sf G_{6,14}}}

\newcommand{\Gq}{{\sf G_{6,15}}}

\newcommand{\lrangle}[1]{\left\langle#1\right\rangle}

\newcommand{\diag}{\operatorname{diag}}
\newcommand{\minors}{\operatorname{minors}}

\newcommand{\SNF}{\operatorname{SNF}}

\title{The characterization of graphs with two trivial distance ideals}

\author{Carlos A. Alfaro\thanks{\texttt{carlos.alfaro@banxico.org.mx}, Banco de M\'exico, Mexico}\qquad Teresa I. Hoekstra-Mendoza\thanks{\texttt{maria.idskjen@cimat.mx}, Centro de Investigaci\'on en Matem\'aticas, Guajanajuato, Mexico}\\ Juan Pablo Serrano\thanks{\texttt{jpserranop@math.cinvestav.mx}, Departamento de Matem\'aticas, Centro de Investigaci\'on y de Estudios Avanzados del IPN}\qquad Ralihe R. Villagr\'an\thanks{\texttt{rvillagran@wpi.edu}, Department of Mathematical Sciences, Worcester Polytechnic Institute, Worcester, USA}}

\date{}

\begin{document}

\maketitle

\begin{abstract}
    The distance ideals of graphs are algebraic invariants that generalize the Smith normal form (SNF) and the spectrum of several distance matrices associated with a graph. 
    In general, distance ideals are not monotone under taking induced subgraphs.
    However, in \cite{at} the characterizations of connected graphs with one trivial distance ideal over $\mathbb{Z}[X]$ and over $\mathbb{Q}[X]$ were obtained in terms of induced subgraphs, where $X$ is a set of variables indexed by the vertices.
    Later, in \cite{alfaro2}, the first attempt was made to characterize the family of connected graphs with at most two trivial distance ideals over $\mathbb{Z}[X]$. 
    There, it was proven that these graphs are $\{{\cal F},\textsf{odd-holes}_{7}\}$-free, where $\textsf{odd-holes}_{7}$ consists of the odd cycles of length at least seven and $\cal F$ is a set of sixteen graphs. 
    Here, we give a characterization of the $\{\mathcal{F},\textsf{odd-holes}_{7}\}$-free graphs and prove that the $\{\mathcal{F},\textsf{odd-holes}_{7}\}$-free graphs are precisely the graphs with at most two trivial distance ideals over $\mathbb{Z}[X]$. 
    As byproduct, we also find that the determinant of the distance matrix of a connected bipartite graph is even, this suggests that it is possible to extend, to connected bipartite graphs, the Graham-Pollak-Lovász celebrated formula $\det(D(T_{n+1}))=(-1)^nn2^{n-1}$, and the Hou-Woo result stating that $\SNF(D(T_{n+1}))=\sf{I}_2\oplus 2\sf{I}_{n-2}\oplus (2n)$, for any tree $T_{n+1}$ with $n+1$ vertices.
    Finally, we also give the characterizations of graphs with at most two trivial distance ideals over $\mathbb{Q}[X]$, and the graphs with at most two trivial distance univariate ideals.
\end{abstract}

\noindent
\textbf{Keywords:}
distance ideals, forbidden induced subgraph, distance matrix, Graham-Pollak-Lovász formula.

\noindent
\textbf{MSC:} 13F70, 05C25, 05C50, 05E99, 13P15, 15A03, 68W30.

\section{Introduction}
In this paper, all graphs will be considered simple and connected.
Let $G=(V,E)$ be a connected graph with $n$ vertices and $X_G=\{x_u \, : \, u\in V(G)\}$ a set of indeterminates.
When the context is clear, we will only use $X$ instead of $X_G$.
The {\it distance} $d_G(u,v)$ between the vertices $u$ and $v$ is the number of edges of a shortest path between them.
We refer the reader to \cite{bondy} for any notion of graph theory not explicitly defined here.
Let $\diag(X)$ denote the diagonal matrix with indeterminates on the diagonal.
The {\it distance matrix} $D(G)$ of $G$ is the matrix with rows and columns indexed by the vertices of $G$ where the $uv$-entry is equal to $d_G(u,v)$.
Thus, the {\it generalized distance matrix} $D_X(G)$ of $G$ is the matrix with rows and columns indexed by the vertices of $G$ defined as $\diag(X)+D(G)$.
Note that we can recover the distance matrix from the generalized distance matrix by evaluating $X$ at the zero vector, that is, $D(G)=D_X(G)|_{X={\bf 0}}$.

Let $\mathfrak{R}[X]$ be the polynomial ring over a commutative ring $\mathfrak{R}$ with unity in the variables $X$.
For all $i\in[n]:=\{1,..., n\}$, the $i$-{\it th} distance ideal $I^\mathfrak{R}_i(G)$ of $G$ is the ideal $\langle \minors_i(D_X(G))\rangle$ over $\mathfrak{R}[X]$, 
where ${\rm minors}_i(D_X(G))$ is the set of determinants of the submatrices of size $i\times i$ of $D_X(G)$.

An ideal is said to be {\it unit} or {\it trivial} if it is equal to $\langle1\rangle$, that is, the entire polynomial ring $\mathfrak{R}[X]=\langle 1 \rangle$. 
Let $\Phi_\mathfrak{R}(G)$ denote the maximum integer $i$ for which $I^\mathfrak{R}_i(G)$ is trivial.
Let $\Lambda^\mathfrak{R}_k$ denote the family of connected graphs with at most $k$ trivial distance ideals over $\mathfrak{R}[X]$.
Note that every connected graph with at least two vertices has at least one trivial distance ideal.

In general, the distance ideals are not monotone under taking induced subgraphs.
However, we have the following results.

\begin{lemma}\cite{at}\label{lemma:inducemonotone}
Let $H$ be an induced subgraph of $G$ such that for every pair of vertices $v_i,v_j$ in $V(H)$, there is a shortest path from $v_i$ to $v_j$ in $G$ which lies entirely in $H$.
Then, $I^\mathfrak{R}_i(H)\subseteq I^\mathfrak{R}_i(G)$  for all $i\leq |V(H)|$.
\end{lemma}

In particular, we have the following.

\begin{lemma}\cite{at}\label{lemma:distance2inducedmonotone}
Let $H$ be an induced subgraph of $G$ with diameter equal to 2, that is, the distance between any pair of vertices in $H$ is at most 2.
Then $I^\mathfrak{R}_i(H)\subseteq I^\mathfrak{R}_i(G)$  for all $i\leq |V(H)|$.
\end{lemma}

Distance-hereditary graphs are another related family, defined by Howorka in \cite{H77}.
A graph $G$ is {\it distance-hereditary} if for each connected induced subgraph $H$ of $G$ and every pair $u$ and $v$ of vertices in $H$, $d_H(u,v)=d_G(u,v)$.

\begin{lemma}\label{lem:dishere}
If $H$ is a connected induced subgraph of a distance-hereditary graph $G$, then $I^\mathfrak{R}_i(H)\subseteq I^\mathfrak{R}_i(G)$  for all $i\leq |V(H)|$.    
\end{lemma}

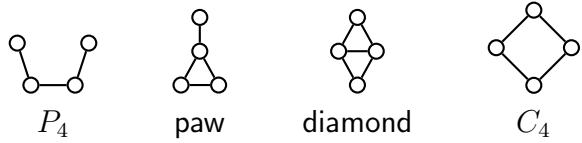
\begin{figure}[ht]
\begin{center}
\begin{tabular}{c@{\extracolsep{10mm}}c@{\extracolsep{10mm}}c@{\extracolsep{10mm}}c@{\extracolsep{10mm}}c}
	\begin{tikzpicture}[scale=.5,thick]
	\tikzstyle{every node}=[minimum width=0pt, inner sep=2pt, circle]
	\draw (126+36:1) node (v1) [draw] {};
	\draw (198+36:1) node (v2) [draw] {};
	\draw (270+36:1) node (v3) [draw] {};
	\draw (342+36:1) node (v4) [draw] {};
	\draw (v1) -- (v2);
	\draw (v2) -- (v3);
	\draw (v4) -- (v3);
	\end{tikzpicture}
&
	\begin{tikzpicture}[scale=.5,thick]
	\tikzstyle{every node}=[minimum width=0pt, inner sep=2pt, circle]
	\draw (-.5,-.9) node (v1) [draw] {};
	\draw (.5,-.9) node (v2) [draw] {};
	\draw (0,0) node (v3) [draw] {};
	\draw (0,.9) node (v4) [draw] {};
	\draw (v1) -- (v2);
	\draw (v1) -- (v3);
	\draw (v2) -- (v3);
	\draw (v3) -- (v4);
	\end{tikzpicture}
&
	\begin{tikzpicture}[scale=.5,thick]
	\tikzstyle{every node}=[minimum width=0pt, inner sep=2pt, circle]
	\draw (-.5,0) node (v2) [draw] {};
	\draw (0,-.9) node (v1) [draw] {};
	\draw (.5,0) node (v3) [draw] {};
	\draw (0,.9) node (v4) [draw] {};
	\draw (v1) -- (v2);
	\draw (v1) -- (v3);
	\draw (v2) -- (v3);
	\draw (v2) -- (v4);
	\draw (v3) -- (v4);
	\end{tikzpicture}
&
    \begin{tikzpicture}[scale=.5,thick]
	\tikzstyle{every node}=[minimum width=0pt, inner sep=2pt, circle]
	\draw (0:1) node (v1) [draw] {};
	\draw (90:1) node (v2) [draw] {};
	\draw (180:1) node (v3) [draw] {};
	\draw (270:1) node (v4) [draw] {};
	\draw (v1) -- (v2) -- (v3) -- (v4) -- (v1);
	\end{tikzpicture}
\\
$P_4$
&
$\sf{paw}$
&
$\sf{diamond}$
&
$C_4$
\end{tabular}
\end{center}
\caption{The graphs $P_4$, $\sf{paw}$, $\sf{diamond}$ and $C_4$.}
\label{fig:forbiddendistance1}
\end{figure}

A graph $G$ is {\it forbidden} for $\Lambda_k^\mathfrak{R}$ if the $(k+1)$-th distance ideal of $G$ is trivial.
In addition, a forbidden graph $H$ for $\Lambda_k^\mathfrak{R}$ is {\it minimal} if $H$ does not contain a connected forbidden graph for $\Lambda_k^\mathfrak{R}$ as induced subgraph, and any graph $G$ containing $H$ as induced subgraph has $G$ is forbidden for $\Lambda_k^\mathfrak{R}$.
The set of minimal forbidden graphs for $\Lambda_k^\mathfrak{R}$ will be denoted by ${\sf Forb}_k^\mathfrak{R}$.
Given a family $\mathcal{G}$ of graphs, a graph $G$ is called $\mathcal{G}${\it -free} if no induced subgraph of $G$ is isomorphic to a member of $\mathcal{G}$.

In \cite{at}, $\Lambda_1^\mathbb{Z}$ was characterized as the $\{P_4,\sf{paw},\sf{diamond}\}$-free graphs; that consists of complete graphs or complete bipartite graphs. 
Also in \cite{at}, $\Lambda_1^\mathbb{Q}$ was characterized as: $\{P_4,{\sf{paw}},{\sf{diamond}}, C_4\}$-free graphs that are star graphs or complete graphs. 
And in \cite{directedideals}, the concept of {\it pattern} was introduced to characterize the digraphs with only one trivial distance ideal over ${\mathbb Z}[X]$.

Some of these forbidden graphs appear in other contexts.
A graph is {\it trivially perfect} if for every induced subgraph the stability number equals the number of maximal cliques.
In \cite[Theorem 2]{G}, Golumbic characterized trivially perfect graphs as $\{P_4,C_4\}$-free graphs. There are other equivalent characterizations of this family, see \cite{CCY,Rubio}.
Therefore, graphs in $\Lambda_1^\mathbb{Q}$ are a subclass of trivially perfect graphs.

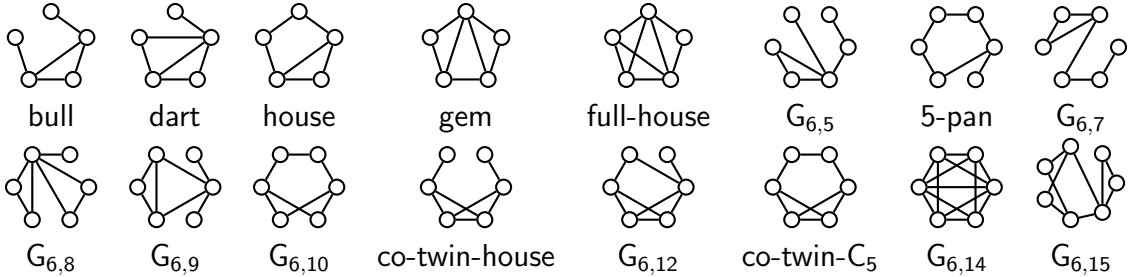
\begin{figure}[ht]
\begin{center}
\begin{tabular}{ccccccccc}
\begin{tikzpicture}[scale=.5,thick]
	\tikzstyle{every node}=[minimum width=0pt, inner sep=2pt, circle]
    \draw (54+36:1) node (v0) [draw] {};
	\draw (126+36:1) node (v1) [draw] {};
	\draw (198+36:1) node (v3) [draw] {};
	\draw (270+36:1) node (v2) [draw] {};
	\draw (342+36:1) node (v4) [draw] {};
    \draw (v0) -- (v4);
    \draw (v1) -- (v3);
    \draw (v2) -- (v3);
    \draw (v2) -- (v4);
    \draw (v3) -- (v4);
\end{tikzpicture}
&
\begin{tikzpicture}[scale=.5,thick]
	\tikzstyle{every node}=[minimum width=0pt, inner sep=2pt, circle]
    \draw (54+36:1) node (v0) [draw] {};
	\draw (126+36:1) node (v1) [draw] {};
	\draw (198+36:1) node (v3) [draw] {};
	\draw (270+36:1) node (v2) [draw] {};
	\draw (342+36:1) node (v4) [draw] {};
    \draw (v0) -- (v4);
    \draw (v1) -- (v3);
    \draw (v1) -- (v4);
    \draw (v2) -- (v3);
    \draw (v2) -- (v4);
    \draw (v3) -- (v4);
\end{tikzpicture}
&
\begin{tikzpicture}[scale=.5,thick]
	\tikzstyle{every node}=[minimum width=0pt, inner sep=2pt, circle]
    \draw (54+36:1) node (v0) [draw] {};
	\draw (126+36:1) node (v1) [draw] {};
	\draw (198+36:1) node (v3) [draw] {};
	\draw (270+36:1) node (v2) [draw] {};
	\draw (342+36:1) node (v4) [draw] {};
    \draw (v0) -- (v1);
    \draw (v0) -- (v4);
    \draw (v1) -- (v3);
    \draw (v2) -- (v3);
    \draw (v2) -- (v4);
    \draw (v3) -- (v4);
\end{tikzpicture}
&
\begin{tikzpicture}[scale=.5,thick]
	\tikzstyle{every node}=[minimum width=0pt, inner sep=2pt, circle]
    \draw (54+36:1) node (v4) [draw] {};
	\draw (126+36:1) node (v1) [draw] {};
	\draw (198+36:1) node (v2) [draw] {};
	\draw (270+36:1) node (v3) [draw] {};
	\draw (342+36:1) node (v0) [draw] {};
    \draw (v0) -- (v3);
    \draw (v0) -- (v4);
    \draw (v1) -- (v2);
    \draw (v1) -- (v4);
    \draw (v2) -- (v3);
    \draw (v2) -- (v4);
    \draw (v3) -- (v4);
\end{tikzpicture}
&
\begin{tikzpicture}[scale=.5,thick]
	\tikzstyle{every node}=[minimum width=0pt, inner sep=2pt, circle]
    \draw (54+36:1) node (v4) [draw] {};
	\draw (126+36:1) node (v1) [draw] {};
	\draw (198+36:1) node (v2) [draw] {};
	\draw (270+36:1) node (v3) [draw] {};
	\draw (342+36:1) node (v0) [draw] {};
    \draw (v0) -- (v3);
    \draw (v0) -- (v4);
    \draw (v1) -- (v2);
    \draw (v1) -- (v3);
    \draw (v1) -- (v4);
    \draw (v2) -- (v3);
    \draw (v2) -- (v4);
    \draw (v3) -- (v4);
\end{tikzpicture}
&
\begin{tikzpicture}[scale=.5,thick]
	\tikzstyle{every node}=[minimum width=0pt, inner sep=2pt, circle]
    \draw (60:1) node (v0) [draw] {};
	\draw (120:1) node (v1) [draw] {};
	\draw (180:1) node (v2) [draw] {};
	\draw (240:1) node (v3) [draw] {};
	\draw (300:1) node (v5) [draw] {};
    \draw (0:1) node (v4) [draw] {};
    \draw (v0) -- (v4);
    \draw (v1) -- (v5);
    \draw (v2) -- (v3);
    \draw (v2) -- (v5);
    \draw (v3) -- (v5);
    \draw (v4) -- (v5);
\end{tikzpicture}
&
\begin{tikzpicture}[scale=.5,thick]
	\tikzstyle{every node}=[minimum width=0pt, inner sep=2pt, circle]
    \draw (60:1) node (v4) [draw] {};
	\draw (120:1) node (v1) [draw] {};
	\draw (180:1) node (v2) [draw] {};
	\draw (240:1) node (v3) [draw] {};
	\draw (300:1) node (v0) [draw] {};
    \draw (0:1) node (v5) [draw] {};
    \draw (v0) -- (v5);
    \draw (v1) -- (v2);
    \draw (v1) -- (v4);
    \draw (v2) -- (v3);
    \draw (v3) -- (v5);
    \draw (v4) -- (v5);
\end{tikzpicture}
&
\begin{tikzpicture}[scale=.5,thick]
	\tikzstyle{every node}=[minimum width=0pt, inner sep=2pt, circle]
    \draw (60:1) node (v5) [draw] {};
	\draw (120:1) node (v1) [draw] {};
	\draw (180:1) node (v2) [draw] {};
	\draw (240:1) node (v3) [draw] {};
	\draw (300:1) node (v4) [draw] {};
    \draw (0:1) node (v0) [draw] {};
    \draw (v0) -- (v4);
    \draw (v1) -- (v2);
    \draw (v1) -- (v5);
    \draw (v2) -- (v5);
    \draw (v3) -- (v4);
    \draw (v3) -- (v5);
\end{tikzpicture}\\
\Ga & \Gb & \Gc & \Gd & \Ge & $\Gf$ & \Gg & $\Gh$\\

\begin{tikzpicture}[scale=.5,thick]
	\tikzstyle{every node}=[minimum width=0pt, inner sep=2pt, circle]
    \draw (60:1) node (v0) [draw] {};
	\draw (120:1) node (v5) [draw] {};
	\draw (180:1) node (v2) [draw] {};
	\draw (240:1) node (v3) [draw] {};
	\draw (300:1) node (v4) [draw] {};
    \draw (0:1) node (v1) [draw] {};
    \draw (v0) -- (v5);
    \draw (v1) -- (v4);
    \draw (v1) -- (v5);
    \draw (v2) -- (v3);
    \draw (v2) -- (v5);
    \draw (v3) -- (v5);
    \draw (v4) -- (v5);
\end{tikzpicture}
&
\begin{tikzpicture}[scale=.5,thick]
	\tikzstyle{every node}=[minimum width=0pt, inner sep=2pt, circle]
    \draw (60:1) node (v0) [draw] {};
	\draw (120:1) node (v4) [draw] {};
	\draw (180:1) node (v2) [draw] {};
	\draw (240:1) node (v3) [draw] {};
	\draw (300:1) node (v1) [draw] {};
    \draw (0:1) node (v5) [draw] {};
    \draw (v0) -- (v5);
    \draw (v1) -- (v5);
    \draw (v2) -- (v3);
    \draw (v2) -- (v4);
    \draw (v3) -- (v4);
    \draw (v3) -- (v5);
    \draw (v4) -- (v5);
\end{tikzpicture}
&
\begin{tikzpicture}[scale=.5,thick]
	\tikzstyle{every node}=[minimum width=0pt, inner sep=2pt, circle]
    \draw (60:1) node (v0) [draw] {};
	\draw (120:1) node (v1) [draw] {};
	\draw (180:1) node (v4) [draw] {};
	\draw (240:1) node (v3) [draw] {};
	\draw (300:1) node (v2) [draw] {};
    \draw (0:1) node (v5) [draw] {};
    \draw (v0) -- (v1);
    \draw (v0) -- (v5);
    \draw (v1) -- (v4);
    \draw (v2) -- (v4);
    \draw (v2) -- (v5);
    \draw (v3) -- (v4);
    \draw (v3) -- (v5);
\end{tikzpicture}
&
\begin{tikzpicture}[scale=.5,thick]
	\tikzstyle{every node}=[minimum width=0pt, inner sep=2pt, circle]
    \draw (60:1) node (v0) [draw] {};
	\draw (120:1) node (v1) [draw] {};
	\draw (180:1) node (v4) [draw] {};
	\draw (240:1) node (v3) [draw] {};
	\draw (300:1) node (v2) [draw] {};
    \draw (0:1) node (v5) [draw] {};
    \draw (v0) -- (v5);
    \draw (v1) -- (v4);
    \draw (v2) -- (v3);
    \draw (v2) -- (v4);
    \draw (v2) -- (v5);
    \draw (v3) -- (v4);
    \draw (v3) -- (v5);
\end{tikzpicture}
&
\begin{tikzpicture}[scale=.5,thick]
	\tikzstyle{every node}=[minimum width=0pt, inner sep=2pt, circle]
    \draw (60:1) node (v0) [draw] {};
	\draw (120:1) node (v1) [draw] {};
	\draw (180:1) node (v4) [draw] {};
	\draw (240:1) node (v3) [draw] {};
	\draw (300:1) node (v2) [draw] {};
    \draw (0:1) node (v5) [draw] {};
    \draw (v0) -- (v5);
    \draw (v1) -- (v4);
    \draw (v1) -- (v5);
    \draw (v2) -- (v3);
    \draw (v2) -- (v4);
    \draw (v2) -- (v5);
    \draw (v3) -- (v4);
    \draw (v3) -- (v5);
\end{tikzpicture}
&
\begin{tikzpicture}[scale=.5,thick]
	\tikzstyle{every node}=[minimum width=0pt, inner sep=2pt, circle]
    \draw (60:1) node (v0) [draw] {};
	\draw (120:1) node (v1) [draw] {};
	\draw (180:1) node (v4) [draw] {};
	\draw (240:1) node (v3) [draw] {};
	\draw (300:1) node (v2) [draw] {};
    \draw (0:1) node (v5) [draw] {};
    \draw (v0) -- (v1);
    \draw (v0) -- (v5);
    \draw (v1) -- (v4);
    \draw (v2) -- (v3);
    \draw (v2) -- (v4);
    \draw (v2) -- (v5);
    \draw (v3) -- (v4);
    \draw (v3) -- (v5);
\end{tikzpicture}
&
\begin{tikzpicture}[scale=.5,thick]
	\tikzstyle{every node}=[minimum width=0pt, inner sep=2pt, circle]
    \draw (60:1) node (v2) [draw] {};
	\draw (120:1) node (v1) [draw] {};
	\draw (180:1) node (v4) [draw] {};
	\draw (240:1) node (v3) [draw] {};
	\draw (300:1) node (v0) [draw] {};
    \draw (0:1) node (v5) [draw] {};
    \draw (v0) -- (v2);
    \draw (v0) -- (v3);
    \draw (v0) -- (v4);
    \draw (v0) -- (v5);
    \draw (v1) -- (v2);
    \draw (v1) -- (v3);
    \draw (v1) -- (v4);
    \draw (v1) -- (v5);
    \draw (v2) -- (v4);
    \draw (v2) -- (v5);
    \draw (v3) -- (v4);
    \draw (v3) -- (v5);
    \draw (v4) -- (v5);
\end{tikzpicture}
&
\begin{tikzpicture}[scale=.5,thick]
	\tikzstyle{every node}=[minimum width=0pt, inner sep=2pt, circle]
    \draw (0:1) node (v0) [draw] {};
	\draw (360/7:1) node (v1) [draw] {};
	\draw (2*360/7:1) node (v4) [draw] {};
	\draw (3*360/7:1) node (v3) [draw] {};
	\draw (4*360/7:1) node (v2) [draw] {};
    \draw (5*360/7:1) node (v5) [draw] {};
    \draw (6*360/7:1) node (v6) [draw] {};
    \draw (v0) -- (v1);
    \draw (v0) -- (v6);
    \draw (v1) -- (v6);
    \draw (v2) -- (v4);
    \draw (v2) -- (v5);
    \draw (v3) -- (v4);
    \draw (v3) -- (v5);
    \draw (v4) -- (v6);
    \draw (v5) -- (v6);
\end{tikzpicture}\\
$\Gi$ & $\Gj$ & $\Gk$ & \Gl & $\Gm$ & $\Gn$ & $\Go$ & $\Gq$\\
\end{tabular}
\end{center}
\caption{Some minimal forbidden graphs for graphs with 2 trivial distance ideals over $\mathbb{Z}[X]$.}
\label{fig:for2}
\end{figure}

In~\cite{alfaro2}, the family $\Lambda_2^\mathbb{Z}$ of graphs with at most two trivial distance ideals over $\mathbb{Z}[X]$ was explored.
In particular, there was found an infinite number of minimal forbidden graphs for $\Lambda_2^\mathbb{Z}$.
Let $\cal F$ be the set of 16 graphs shown in Figure~\ref{fig:for2}.
Specifically, in~\cite{alfaro2}, it was proved that graphs in $\Lambda_2^\mathbb{Z}$ are $\{\mathcal{F},\text{\sf odd-holes}_7\}$-free graphs, where $\text{{\sf odd-holes}}_7$ are cycles of odd length greater or equal than 7.

\begin{lemma}\cite[Theorem 23]{alfaro2}
    Graphs in $\Lambda_2^\mathbb{Z}$ are $\{\mathcal{F},\text{\sf odd-holes}_7\}$-free.
\end{lemma}

In \cite[Theorem 3]{HW}, it was proved that the distance matrix of trees has exactly 2 invariant factors equal to 1, in Section~\ref{sec:SNF_DI} we will explain this point.
Therefore,
\[
\textsf{trees} \subseteq \Lambda_2^\mathbb{Z} \subseteq \{\mathcal{F},\text{\sf odd-holes}_7\} \text{-free graphs}.
\]

Among the forbidden graphs for $\Lambda_2^\mathbb{Z}$, there are several graphs studied in other contexts, like $\Ga$ and $\text{{\sf odd-holes}}_7$ studied in \cite{CI,CII} and \cite{CS}, respectively.
Distance-hereditary graphs are $\{\Gc,\Gd,{\sf domino}, {\sf holes}_5\}$-free graphs, where ${\sf holes}_5$ are cycles of length greater than or equal to 5, which intersect with $\Lambda_2^\mathbb{Z}$.
Another related family is the {\sf 3-leaf powers} that was characterized in \cite{DGHN} as $\{\Ga,\Gb,\Gd\}$-free chordal graphs.

Previously, an analogous notion to the distance ideals but for the adjacency and Laplacian matrices was explored.
These were called {\it critical ideals}, see \cite{corrval}.
They have been explored in \cite{alfacorrval,alfaval,alfaval1,AVV}, and in \cite{alfaro,alflin} new connections have been found in contexts different from the Smith group or the Sandpile group, such as the zero-forcing number and the minimum rank of a graph.
In this setting, the set of forbidden graphs for the family with at most $k$ trivial critical ideals is conjectured to be finite; see \cite[Conjecture 5.5]{alfaval}.
It is interesting that for distance ideals this is not true. 
On the other hand, univariate ideals have been studied as well; see \cite{distinguish,absv1} for instance.

In this paper, we complete the work started in \cite{alfaro2} by proving that $\Lambda_2^\mathbb{Z}$ consists of the $\{\mathcal{F},\text{\sf odd-holes}_7\}$-free graphs, and we will give a description of these graphs.
In particular, we will prove that the {\sf bipartite graphs} are in $\Lambda_2^\mathbb{Z}$, where {\sf bipartite graphs} refer only to the connected ones.
We will do this by proving that the third invariant factor of the distance matrix of {\sf bipartite graphs} is not unity.
In fact, we will be able to prove that the third invariant factor of the distance matrix of {\sf bipartite graphs} is an even number.
This has an interesting additional consequence, the determinant of the distance matrix of bipartite graphs is even.
Therefore, this suggests that it is possible to extend Graham-Pollak-Lovász celebrated formula $\det(D(T_{n+1}))=(-1)^nn2^{n-1}$ \cite{GP} to bipartite graphs, as well as, Hou-Woo \cite{HW} result stating that the Smith normal form of $D(T_{n+1})$ is $\sf{I}_2\oplus 2\sf{I}_{n-2}\oplus (2n)$, for any tree $T_{n+1}$ with $n+1$ vertices. 

The manuscript is organized as follows.
Section~\ref{sec:SNF_DI} provides a comprehensive introduction to the Smith Normal Form (SNF) of integer matrices, paying particular attention to the SNF of distance matrices. 
Proposition~\ref{prop:eval1} serves as the bridge connecting the generalized distance matrix, and the invariant factor of the distance matrices aside, it will play a crucial role in Section~\ref{section:characterization}.
One of the central results of this article is presented in Section~\ref{sec:F odd holes free graphs}, where Theorem~\ref{theo:main theorem2} provides a description of ${\mathcal{F},\text{\sf odd-holes}_7}$-free graphs, thus extending the findings of Alfaro in~\cite{alfaro2}, note that any graph in $\Lambda_2^{\mathbb Z}$ has to be ${\mathcal{F},\text{\sf odd-holes}_7}$-free.  
The complete characterization of $\Lambda_{2}^{\mathbb{Z}}$ will be finalized in Section~\ref{section:characterization}, where it is established that the graphs identified in Theorem~\ref{theo:main theorem2} possess a nontrivial third distance ideal and thus are precisely the graphs in $\Lambda_2^{\mathbb{Z}}$, see Theorem~\ref{theo:main simple}.
Moreover, in that section, it is proved that the third invariant factor of any bipartite graph is $2$ or $4$. Therefore, their determinant is an even number. 
In Section~\ref{sec:5}, the characterization of $\Lambda_{2}^{\mathbb{Q}}$, that is, the graphs with at most two trivial distance ideals with coefficients over rational numbers, is obtained.
Finally, in Section~\ref{sec:6}, the graphs with at most two trivial distance univariate ideals are characterized.

\section{Distance ideals and Smith normal form of distance matrices}\label{sec:SNF_DI}

The Smith normal form (SNF) has been useful in understanding the algebraic properties of combinatorial objects; see \cite{stanley}.
For example, computing the Smith normal form of the adjacency or Laplacian matrix is a standard technique used to determine the Smith group and the critical group of a graph; see \cite{alfaval0,merino,rushanan}.
In fact, Stanley recently commented on the role of the SNF in combinatorics \cite{interview}: ``{\it Although I enjoy these SNF problems, they seem to be mostly problems in algebra, not combinatorics. An exception is the connection between the SNF of the Laplacian matrix of a graph $G$ and chip-firing on $G$... It would be great to have some further combinatorial applications of SNF}''.

Let us recall that two matrices $ M$ and $ N$ of the same order with entries over a commutative ring are {\it equivalent} if there exist unimodular matrices $ P$ and $ Q$ such that ${ M} = { PNQ}$.
Therefore, if $M$ and $N$ are equivalent, then $M$ can be transformed into $N$ by means of the following operations:
\begin{enumerate}
  \item swap any two rows or any two columns.
  \item add an integer multiple of one row to another row.
  \item add an integer multiple of one column to another column.
  \item multiply any row or column by $\pm 1$.
\end{enumerate}
Thus, if $M$ is a square integer matrix, the {\it Smith normal form} of $M$ is the unique diagonal matrix $\diag(f_1,f_2,\dots,f_r,0,\dots,0)$ equivalent to $M$, whose diagonal elements are nonnegative and satisfy that $f_i$ divides $f_{i+1}$, and $r$ is the rank of $M$. 
The elements $f_1,\dots,f_r$ are known as \emph{invariant factors} of $M$.
Kannan and Bachem found in \cite{KB} polynomial algorithms to compute the Smith normal form of an integer matrix.
An alternative way to obtain the Smith normal form is as follows.
Let $\Delta_i(M)$ denote the {\it greatest common divisor} of the $i$-minors of the integer matrix $M$, then the $i$-{\it th} invariant factor, $f_i=f_i(M)$, is equal to $\Delta_i(M)/ \Delta_{i-1}(M)$, where $\Delta_0(M)=1$.
It is known that the Smith normal form may not exist in the ring $\mathbb{Z}[X]$, this is because $\mathbb{Z}[X]$ is not a {\it principal ideal domain} (PID), for example, the ideal $\lrangle{2,x}$ is not principal.

Little is known about the Smith normal forms of distance matrices.
In \cite{HW}, the Smith normal forms of the distance matrices were determined for trees, wheels, cycles, and complements of cycles and are partially determined for complete multipartite graphs.
In \cite{BK}, the Smith normal form of the distance matrices of the
unicyclic graphs and of the wheel graph with trees attached to each
vertex were obtained.

Let ${\bf d}\in \mathbb{Z}^{n}$, the following observation will give us the relation between the Smith normal form of the integer matrix $D_X(G)|_{X={\bf d}}$ and the distance ideals of $G$.
\begin{proposition}\cite{at}\label{prop:eval1}
Let ${\bf d}\in \mathbb{Z}^{n}$.
If $f_1\mid\cdots\mid f_{r}$ are the invariant factors of the integer matrix $D_X(G)|_{X={\bf d}}$, then
\[
I_i^\mathbb{Z}(G)|_{X=\bf d}=\left\langle \prod_{j=1}^{i} f_j \right\rangle=\left\langle \Delta(D_X(G)|_{X={\bf d}}) \right\rangle\text{ for all }1\leq i\leq r.
\]
\end{proposition}
Thus, to recover $\Delta_i(D(G))$ and the invariant factors of $D(G)$ from the distance ideals, we only need to evaluate $I_i^\mathbb{Z}(G)$ at $X_G={\bf 0}$.
Therefore, if the distance ideal $I^\mathbb{Z}_i(G)$ is trivial, then $\Delta_i(D(G))$ and $f_i(D(G))$ are equal to $1$.
Equivalently, if $\Delta_i(D(G))$ and $f_i(D(G))$ are not equal to $1$, then the distance ideal $I_i^\mathbb{Z}(G)$ is not trivial.
For example, in \cite{HW}, it was shown that $\SNF(D(T_{n+1}))=\sf{I}_2\oplus 2\sf{I}_{n-2}\oplus (2n)$, thus $I_3^\mathbb{Z}(T_{n+1})$ is not trivial.
Therefore, $\textsf{trees} \subseteq \Lambda_2^\mathbb{Z}$.

The following form of Proposition~\ref{prop:eval1} will be used later.

\begin{lemma}\label{lem:deltaideal}
    If $\Delta_i(D_X(G)|_{X={\bf d}})\neq 1$ for some ${\bf d}\in \mathbb{Z}^{n}$, then the distance ideal $I_i^\mathbb{Z}(G)$ is not trivial.
\end{lemma}

Let $\phi(M)$ denote the number of invariant factors of the matrix $M$ equal to 1.
The following result is a direct consequence.

\begin{corollary}\cite{at}\label{coro:eval1}
    Let ${\bf d}\in \mathbb{Z}^{n}$.
    For any graph $G$, $\Phi(G)\leq \phi(D_X(G)|_{X={\bf d}})$.
    And, for any positive integer $k$, $\Lambda_k^\mathbb{Z}$ contains the family of graphs with $\phi(D_X(G)|_{X={\bf d}})\leq k$.
\end{corollary}

In particular, finding a characterization of $\Lambda_k^\mathbb{Z}$ can help in the characterization of the graphs with $\phi(D(G))\leq k$.

\section{$\{\mathcal{F},\text{\sf odd-holes}_7\}$-free graphs}\label{sec:F odd holes free graphs}

Despite the fact that distance ideals are not induced monotone, in \cite{at} classifications of the graphs in $\Lambda_1^{\mathbb Z}$ and $\Lambda_1^{\mathbb Q}$ were obtained in terms of forbidden induced subgraphs.
We will recall them in order to obtain a classification of $\Lambda_2^{\mathbb Z}$.
As stated previously, we will only consider connected graphs.

The next classifications were obtained in \cite{at}, and it is worth mentioning that in \cite{alfaro2} an infinite number of minimal forbidden graphs were found for graphs with at most two trivial distance ideals over $\mathbb Z[X]$.

\begin{theorem}\cite{at}\label{teo:classification}
Let $G$ be a connected graph, the following are equivalent:
\begin{enumerate}
\item $G$ has only 1 trivial distance ideal over $\mathbb{Z}[X]$.
\item $G$ is $\{P_4,\sf{paw},\sf{diamond}\}$-free.
\item $G$ is an induced subgraph of $K_{m,n}$ or $K_{n}$.
\end{enumerate}
\end{theorem}

\begin{theorem}\cite{at}\label{teo:classification2}
Let $G$ be a connected graph, the following are equivalent:
\begin{enumerate}
\item $G$ has only 1 trivial distance ideal over $\mathbb{Q}[X]$.
\item $G$ is $\{P_4,\sf{paw},\sf{diamond}, C_4\}$-free.
\item $G$ is an induced subgraph of $K_{1,n}$ or $K_{n}$.
\end{enumerate}
\end{theorem}

Later, in \cite{alfaro2} it was found that the graphs in $\Lambda_2^\mathbb{Z}$ are $\{\mathcal{F},\text{\sf odd-holes}_7\}$-free.
This, combined with the fact that the distance matrix of a tree has exactly 2 invariant factors equal to 1, shows that
\[
\textsf{trees} \subseteq \Lambda_2^\mathbb{Z} \subseteq \{\mathcal{F},\text{\sf odd-holes}_7\} \text{-free graphs}.
\]

In this section, we will give a complete characterization of the $\{\mathcal{F},\text{\sf odd-holes}_7\}$-free graphs. 
Let us begin the discussion recalling the next result, which follows since the $\text{\sf anti-hole}_{6}$ graphs $\left(\overline{C}_n\text{ for }n\geq6\right)$ contain an induced {\sf house}.

\begin{lemma}\label{lem:housefreeantioddhole}
    Let $G$ be a graph. If $G$ is a {\sf house}-free graph, then it is $\text{\sf anti-hole}_{6}$-free.  
\end{lemma}

On the other hand, if $G$ is $\{\mathcal{F},\text{\sf odd-holes}_7\}$-free and has an induced $C_5$, then any vertex $v\in V(G)\setminus V(C_5)$ adjacent to any subset of vertices in $V(C_5)$ must not be adjacent to only one vertex, otherwise {\sf 5-pan} would be an induced subgraph of $G$. 
Also, the vertex $v$ cannot be adjacent to two vertices in $V(C_5)$, otherwise $G$ would have either a {\sf bull} graph or a $\sf G_{6,10}$ as induced subgraph.
Similarly, $v$ cannot be adjacent to three vertices in $V(C_5)$ since this would induce a {\sf bull} graph or a $\Gn$ graph in $G$. 
Finally, if $|N_G(v)\cap V(C_5)|\geq 4$, then $G$ contains a {\sf gem} graph as an induced subgraph, which is not possible. 
Thus, next results follow. 

\begin{lemma}\label{lem:householes6free}
    If $G$ is a connected $\{\mathcal{F},\text{\sf odd-holes}_7\}$-free and has an induced $C_5$, then $G$ must be a $C_5$.
\end{lemma}

A graph is {\it perfect} if every induced subgraph satisfies that the size of its largest clique equals its minimal number of colors needed to color the vertices in such a way that no two adjacent vertices have the same color.
Let $\text{\sf odd-holes}_5$ be the cycles of odd length of size at least 5.

\begin{theorem}[Strong perfect graph theorem]\cite{perfect}
    A graph is perfect if and only if neither $G$ nor $\overline{G}$ contains $\text{\sf odd-holes}_5$ as induced subgraphs. 
\end{theorem}

It is interesting that any connected $\{\mathcal{F},\text{\sf odd-holes}_7\}$-free, other than $C_5$, is included in this special class of graphs.

\begin{corollary}\label{lemma:perfect}
    Let $G$ be a connected $\{\mathcal{F},\text{\sf odd-holes}_7\}$-free graphs such that $G\neq C_5$, then $G$ is a perfect graph.
\end{corollary}
\begin{proof}
    This follows from Lemmas~\ref{lem:housefreeantioddhole} and~\ref{lem:householes6free}, and from the {\it strong perfect graph theorem}.
\end{proof}

We will need a couple of results. 
\begin{lemma}\cite[Theorem 1]{O88}\label{lemma:pawfree}
  Let $G$ be a connected {\sf paw}-free graph.
  Then $G$ is either {\sf triangle}-free or a complete multipartite graph.
\end{lemma}

A {\it chord} in a cycle is an edge between two vertices in the cycle that is not an edge of the cycle.
A graph $G$ is {\it chordal} if each cycle in $G$ of length at least 4 has at least one chord.
Let $\text{\sf holes}_4$ denote the cycles of length at least 4.
Then, $G$ is chordal if and only if $G$ is $\text{\sf holes}_4$-free.

\begin{lemma}[\cite{nishimura,rautenbach}]\label{lemma:3PL}
    Let $G$ be a connected $\{\textsf{bull, dart, gem}\}$-free and chordal graph. 
    Then it is a {\sf 3-leaf power} graph. That is, a graph whose vertices are the leaves of a tree and whose edges connect a pair of leaves whose distance in the tree is at most $3$. Equivalently, a graph that results from a tree $T$ by replacing every vertex by a clique of arbitrary size.
\end{lemma}
It is worth mentioning that, in \cite{linear3leaf}, it was shown that {\sf 3-leaf power} graphs can be recognized in linear time.

Let $\mathbf{a}\in\mathbb{Z}^k$ be an integer vector, and let us define $P_k^{\mathbf{a}}$ as the graph obtained from $P_k$ by replacing the $i$-th vertex for a clique of size $a_i+1$ if $a_i> 0$, or by an independent set of size $-a_i +1$ if $a_i <0$ and respecting the adjacency prescribed by $P_k$. If $a_i=0$ we keep the $i$-th vertex. For instance, if $m_1,n_1,m_2,n_2,m_3\geq 0$, then $P_5^{(-m_1,n_1,-m_2,n_2,-m_3)}$ is the following graph
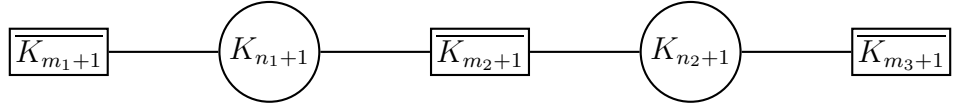
\begin{center}
        \begin{tikzpicture}[scale=1.4,thick]
	\tikzstyle{every node}=[minimum width=0pt, inner sep=2pt, circle]
	\draw (-4,0) node (v1) [draw, rectangle] {$\overline{K_{m_1+1}}$}; 
	\draw (-2,0) node (v2) [draw] {$K_{n_1+1}$};
	\draw (0,0) node (v3) [draw, rectangle] {$\overline{K_{m_2+1}}$};
	\draw (2,0) node (v4) [draw] {$K_{n_2+1}$};
        \draw (4,0) node (v5) [draw,rectangle] {$\overline{K_{m_3+1}}$};
	\draw (v1) -- (v2) -- (v3) -- (v4) -- (v5);
	\end{tikzpicture} 
        \end{center}
where each edge means that all the vertices on one side are adjacent to every vertex on the other. 
Now we are ready to prove the main result of this section.
\begin{theorem}\label{theo:main theorem2}
    Let $G$ be a connected $\{{\cal F},\text{\sf odd-holes}_7\}$-free graph with $n$ vertices. Then $G$ is one of the following:
    \begin{itemize}
        \item[i)] $C_5$;
        \item[ii)] a bipartite graph; 
        \item[iii)] a complete tripartite graph;
        \item[iv)] $K_{n-p+1,1,\ldots,1}$ where $p$ is the number of partitions;
        \item[v)] a connected induced subgraph of $P_5^{(-m_1,n_1,-m_2,n_2,-m_3)}$ with $m_1,n_1,m_2,n_2,m_3 \geq 0$;

        \item[vi)] or a connected induced subgraph of $P_4^{(n_1,-m_1,-m_2,n_2)}$ with $n_1,m_1,n_2,m_2 \geq 0$.

    \end{itemize}
\end{theorem}

\begin{proof}
    By Lemma~\ref{lem:householes6free}, if $C_5$ is an induced subgraph of $G$, then $G=C_5$, that is, case (i). Thus, let us assume that $G$ is $C_5$-free. Moreover, if $G$ is {\sf triangle}-free, then $G$ is a bipartite graph and note that it is also $\cal F$-free, that is, case (ii).

    Therefore, let us assume that $G$ has an induced $K_3$. 
    Moreover, assuming that $G$ is {\sf paw}-free, then Lemma~\ref{lemma:pawfree} implies that $G$ is a complete multipartite graph. 
    However, in this case, if we have at least 4 partitions in $G$, we need all partitions but one to consist of a single vertex, since ${\sf G_{6,14}}=K_{2,2,1,1}$ is forbidden. Thus, $G$ is a complete tripartite graph, note that it is $\cal F$-free and \textsf{odd-holes}-free, or a complete multipartite graph $G=K_{n-p+1,1,\ldots,1}$ for some $n\geq p$, where $p$ is the number of partitions of $G$, these fall into cases (iii) and (iv), respectively.
    
    Hence, now assume that $G$ has an induced {\sf paw} with vertices $v_1,v_2,v_3$ and $v_4$, such that $v_1,\ v_2$, and $v_3$ form a {\sf triangle}, and that $v_4$ is adjacent to $v_3$ but not to $v_1$ nor $v_2$, see Figure~\ref{fig:pawinG}. 

\begin{figure}[ht]
\centering
\begin{tikzpicture}[scale=1.4,thick]
	\tikzstyle{every node}=[minimum width=0pt, inner sep=2pt, circle]
	\draw (-.5,-.9) node (v1) [draw] {\tiny $v_1$};
	\draw (.5,-.9) node (v2) [draw] {\tiny $v_2$};
	\draw (0,0) node (v3) [draw] {\tiny $v_3$};
	\draw (0,.9) node (v4) [draw] {\tiny $v_4$};
	\draw (v1) -- (v2);
	\draw (v1) -- (v3);
	\draw (v2) -- (v3);
	\draw (v3) -- (v4);
	\end{tikzpicture}
 \caption{An induced \textsf{paw} graph in $G$}
 \label{fig:pawinG}
\end{figure}
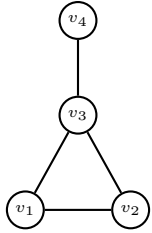

We might assume $G$ is not a complete graph, since this case was already considered as a complete multipartite graph with $n=p$.

\begin{claim}\label{claim0}
    Let $G$ be a $\cal F$-free and \textsf{odd-hole}-free graph with at least 5 vertices and an induced \textsf{paw}. Then every vertex in $G$ not in the \textsf{paw} must be adjacent to it. 
\end{claim}
\begin{proof}
    There must be a vertex $u$ adjacent to \textsf{paw} since $G$ is connected and $|G|\geq 5$. Notice that $V(\textsf{paw})\cup \{u\}$ induces a $P_3^{(1,-1,0)}$, a $P_3^{(2,0,0)}$, a $P_4^{(0,1,0,0)}$, a $P_3^{(1,0,1)}$, a $P_3^{(1,0,-1)}$, or a $P_4^{(1,0,0,0)}$.
    Now, assume there is another vertex $z$ non-adjacent to $V(\textsf{paw})$. Then, considering a path from $z$ to \textsf{paw}, by connectivity, there is a vertex $u$ adjacent to \textsf{paw} and a vertex $w$ non-adjacent to \textsf{paw} but adjacent to $u$. Then $V(\textsf{paw})\cup \{u,w\}$ induces a $G_{6,12}$, a bull graph, a co-twin-house, a bull graph, a $G_{6,5}$, or a $G_{6,7}$. Thus, every vertex in $G$ is adjacent to \textsf{paw}.
\end{proof}

In particular, with these assumptions, the previous claim implies that
\begin{itemize}
    \item[a)] $G$ has diameter at least $2$ and at most $4$.
\end{itemize}

Furthermore, let us prove the following fact.

\begin{claim}\label{claim1}
    It is impossible to have an induced paw and an induced cycle of order six at the same time. 
\end{claim}

\begin{proof}
    Let us assume that there is an induced $C_6$ with vertices $V(C_6)=\{c_1,c_2,c_3,c_4,c_5,c_6\}$ and adjacent accordingly to the vertices order. First note that: 
    \begin{itemize}
        \item[b)] Let $v\notin V(C_6)$ be adjacent to the cycle $C_6$, then $|E(v,C_6)| \leq 3$. Moreover, if $|E(v,C_6)|= 3$, then $$E(v,C_6)=\{(v,c_1), (v,c_3),(v,c_5)\} \text{ or } \{(v,c_2),(v,c_4),(v,c_6)\}.$$ 
    If $|E(v,C_6)|= 2$ then $E(v,C_6)=\{(v,c_i), (v,c_{i+2})\}$ for $i=1,2,3,4$, or $E(v,C_6)=\{(v,c_1), (v,c_{5})\}$ or $E(v,C_6)=\{(v,c_2), (v,c_{6})\}$ ($E(v,C_6)\neq \{(v,c_1), (v,c_{6})\}$ or it forms an induced bull graph). In particular $v$ is not adjacent to a pair of adjacent vertices in $C_6$
        \item[c)] Also, note that a pair of adjacent vertices not in $C_6$ can not be adjacent with the same vertex in $C_6$. 
        
    \end{itemize}
    
    Now let $A=V(\textsf{paw})\cap V(C_6)$, $|A|\leq 3$ as there are no triangles in $C_6$. Also, $|A|\neq 3$ by (b) and similarly $|A|\neq 2$ by (b) and (c).
    
    Now, if $|A| =1$, then $A$ can only consist of $v_4$ by (b). 
    Assume, without loss of generality, that $v_4=c_1$. Then $v_3$ is not adjacent to $c_2$ and is not adjacent to $c_6$, by (b). Moreover, $c_4$ is not adjacent to $v_3$ because otherwise $\{v_3,c_1,c_6,c_5,c_4\}$ induces a $C_5$ or a house graph. Then, $c_4$ must be adjacent to both $v_1$ and $v_2$. However, this is not allowed by (c). 
    
    On the other hand, if $|A|=0$, first note that if $v_1$ ($v_2$) is adjacent to a vertex in $C_6$, then $v_2$ ($v_1$) and $v_3$ are not adjacent to the same vertex by (c). Thus, inducing a bull graph or a house graph. Similarly, if $v_3$ is adjacent to a vertex in $C_6$ we have an induced $G_{6,7}$ or an induced bull graph or an induced House graph, or an induced $G_{6,5}$, or an induced $G_{6,15}$. 
    If $v_4$ is adjacent to $C_6$, we have an induced $G_{6,7}$ or some other vertex in \textsf{paw} have to be adjacent to a vertex in $C_6$ by (b), and we have already proved that this forces a forbidden structure. 
    Thus, if $|A|=0$, then $E(\textsf{paw},P_6)=\emptyset$, which is impossible by Claim \ref{claim0}. 
    
\end{proof}

Thus, we conclude that since $G$ has an induced \textsf{paw} and is $C_5$-free, then $G$ is $P_6$-free and \textsf{holes}-free.
If we further assume that $G$ is \textit{chordal}. Then, it is known that the $\{\textsf{bull, dart, gem}\}$-free chordal graphs are precisely the {\sf 3-leaf power} graphs according to the Lemma~\ref{lemma:3PL}. 

Now, from the set of forbidden graphs $\cal F$, we have the following six forbidden structures for the {\sf 3-leaf power} graph $G$, see Figure~\ref{fig:3LP}, where edges mean that all vertices in the clique are adjacent to all vertices in the adjacent clique (edge).
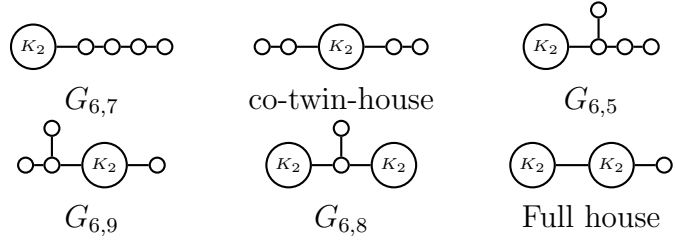
\begin{figure}[ht]
\centering
\begin{tabular}{c@{\extracolsep{1cm}}c@{\extracolsep{1cm}}c@{\extracolsep{1cm}}c}
 \begin{tikzpicture}[scale=1.4,thick]
	\tikzstyle{every node}=[minimum width=0pt, inner sep=2pt, circle]
	\draw (-1,0) node (v1) [draw] {\tiny $K_2$};
	\draw (-.5,0) node (v2) [draw] {};
	\draw (-.25,0) node (v3) [draw] {};
	\draw (-0,0) node (v4) [draw] {};
        \draw (0.25,0) node (v5) [draw] {};
	\draw (v1) -- (v2) -- (v3) -- (v4) -- (v5);
\end{tikzpicture}    
 &
 \begin{tikzpicture}[scale=1.4,thick]
 \tikzstyle{every node}=[minimum width=0pt, inner sep=2pt, circle]
	\draw (-1,0) node (v1) [draw] {};
	\draw (-.75,0) node (v2) [draw] {};
	\draw (-.25,0) node (v3) [draw] {\tiny $K_2$};
	\draw (0.25,0) node (v4) [draw] {};
        \draw (0.5,0) node (v5) [draw] {};
	\draw (v1) -- (v2) -- (v3) -- (v4) -- (v5);
\end{tikzpicture}   
 &
 \begin{tikzpicture}[scale=1.4,thick]
 \tikzstyle{every node}=[minimum width=0pt, inner sep=2pt, circle]
	\draw (-1,0) node (v1) [draw] {\tiny $K_2$};
	\draw (-.5,0) node (v2) [draw] {};
	\draw (-.25,0) node (v3) [draw] {};
	\draw (0,0) node (v4) [draw] {};
        \draw (-.5,0.35) node (v5) [draw] {};
	\draw (v1) -- (v2) -- (v3) -- (v4);
        \draw (v2) -- (v5);
\end{tikzpicture} \\
    $G_{6,7}$ & co-twin-house & $G_{6,5}$\\
\begin{tikzpicture}[scale=1.4,thick]
 \tikzstyle{every node}=[minimum width=0pt, inner sep=2pt, circle]
	\draw (-1,0) node (v1) [draw] {};
	\draw (-.75,0) node (v2) [draw] {};
	\draw (-.25,0) node (v3) [draw] {\tiny $K_2$};
	\draw (0.25,0) node (v4) [draw] {};
        \draw (-.75,0.35) node (v5) [draw] {};
	\draw (v1) -- (v2) -- (v3) -- (v4);
        \draw (v2) -- (v5); 
\end{tikzpicture}    
&
\begin{tikzpicture}[scale=1.4,thick]
 \tikzstyle{every node}=[minimum width=0pt, inner sep=2pt, circle]
	\draw (-1,0) node (v1) [draw] {\tiny $K_2$};
	\draw (-.5,0) node (v2) [draw] {};
	\draw (0,0) node (v3) [draw] {\tiny $K_2$};
	\draw (-.5,.35) node (v4) [draw] {};
	\draw (v1) -- (v2) -- (v3);
        \draw (v2) -- (v4);
\end{tikzpicture}  
& 
\begin{tikzpicture}[scale=1.4,thick]
 \tikzstyle{every node}=[minimum width=0pt, inner sep=2pt, circle]
	\draw (-1,0) node (v1) [draw] {\tiny $K_2$};
	\draw (-.25,0) node (v2) [draw] {\tiny $K_2$};
	\draw (0.25,0) node (v3) [draw] {};
	\draw (v1) -- (v2) -- (v3);
	
 \end{tikzpicture}\\
    $G_{6,9}$ & $G_{6,8}$ & Full house\\
\end{tabular}

 \caption{Forbidden structures of the {\sf 3-leaf power} graph.}
 \label{fig:3LP}
\end{figure}

Therefore, by (a), let us focus on graphs with diameter at most 4 and at least 2. 
If $P_5$ is an induced subgraph of $G$, all possible {\sf 3-leaf power} graphs with these forbidden structures and with an induced \textsf{paw} are precisely graphs of the form $P_5^{(-m_1,n_1,0,n_2,-m_2)}$ with $m_1,n_1,n_2,m_2\geq 0$ and $n_1+n_2\geq 1$ by the induced \textsf{paw}. 
This is included in case of (v).
If there is no induced $P_5$ but there is an induced $P_4$ we have the induced subgraphs of the previous form without induced $P_5$'s (subcase of (v)) and the graphs of the form $P_4^{(n_1,0,0,n_2)}$ with $n_1,n_2\geq 0$ and $n_1+n_2\geq 1$.
This is included in case (vi). On the other hand, if the diameter is 2, we have the following sub-cases of (v) and (vi); $P_3^{(n_1,0,n_2)}$ with $n_1,n_2\geq 0$ and $n_1+n_2\geq 1$, and $P_3^{(n,0,-m)}$ with $n\geq 1$ and $m\geq 0$.

Henceforth, let us assume that $C_4=\{c_1,c_2,c_3,c_4\}$ is an induced subgraph of $G$ as well as the \textsf{paw} graph.

First, for diameter 2, starting from a maximal path $P_3$ with vertices $\{p_1,p_2,p_3\}$, any vertex $v$ adjacent to it must not be adjacent only to $p_1$ or $p_3$. Also, notice that if there is a vertex $u$ adjacent to $v$ but not to $P_3$, then we have an induced $P_4$, a bull graph, or a dart graph. Thus, any vertex in $G$ must be adjacent to $P_3$. 
Therefore, considering the induced $C_4$, any other vertex must be adjacent to at least two vertices in $C_4$ (if it is adjacent to only two vertices in $C_4$ then these vertices are not adjacent in $C_4$ or a House graph is induced in $G$). If there is a vertex $v$ adjacent to every vertex in $C_4$ then if there is another vertex $u$ adjacent to $2$, or $3$ vertices in $C_4$ we get an induced dart, a $P_4$ (actually a gem graph) or a full house graph. If $u$ is adjacent to the four vertices in $C_4$, then it is not adjacent to $v$ or we get an induced $K_{2,2,1,1}$. However, this yields a complete tripartite graph, so there is no induced \textsf{paw}. Hence, we can assume that any vertex not in $C_4$ is adjacent to 2 or 3 vertices in $C_4$.
Moreover, there is at least one vertex adjacent to 3 vertices in $C_4$, say $v$, to ensure that we have induced \textsf{paw} in $G$. That is, $\{c_1,c_2,c_3,c_4,v\}$ induces $P_3^{(1,-1,0)}$, therefore any other vertex $u$ adjacent to 2 vertices in $C_4$ must induce $P_3^{(1,-1,-1)}$. On the other hand, if $u$ is adjacent to 3 vertices in $C_4$ it must induce $P_3^{(1,-1,1)}$ and in these two cases we must have $P_3^{(n,-m_1,-m_2)}$ or $P_3^{(n_1,-m,n_2)}$ with $n,n_1,m_1,m\geq 2$ and $m_2,n_2\geq 1$, respectively, which are included in cases (v) and (vi): 
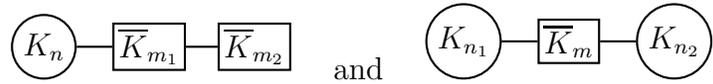
\begin{center}
    \begin{tabular}{ccc}
      \begin{tikzpicture}[scale=1.4,thick]
\tikzstyle{every node}=[minimum width=0pt, inner sep=2pt, circle]
	\draw (-1,0) node (v2) [draw] {$K_{n}$};
	\draw (0,0) node (v3) [draw, rectangle] {$\overline{K}_{m_1}$};
	\draw (1,0) node (v4) [draw,rectangle] {$\overline{K}_{m_2}$};
	\draw (v2) -- (v3) -- (v4);
\end{tikzpicture}   &\text{ and }& \begin{tikzpicture}[scale=1.4,thick]
\tikzstyle{every node}=[minimum width=0pt, inner sep=2pt, circle]
	\draw (-1,0) node (v2) [draw] {$K_{n_1}$};
	\draw (0,0) node (v3) [draw, rectangle] {$\overline{K}_{m}$};
	\draw (1,0) node (v4) [draw] {$K_{n_2}$};
	\draw (v2) -- (v3) -- (v4);
\end{tikzpicture}. \\
         
    \end{tabular}
\end{center}

Second, for diameter 3, starting from a maximal path $P_4$ with vertices $\{p_1,p_2,p_3,p_4\}$. Then any vertex $v$ adjacent to $P_4$ is adjacent to $\{p_2\}$, $\{p_3\}$, $\{p_1,p_2\}$, $\{p_3,p_4\}$, $\{p_1,p_3\}$, $\{p_2, p_4\}$, $\{p_1,p_2,p_3\}$ or $\{p_2,p_3,p_4\}$. Note that this generates a $P_4^{\mathbf{a}}$ where $\mathbf{a}$ has all zero entries except one entry equal to $\pm 1$.

\begin{claim}
    There must be a vertex $v$ adjacent to $P_4$ such that $\{p_1,p_2,p_3,p_4,v\}$ induces at least a triangle.
\end{claim}
\begin{proof}
    Suppose that there is no such vertex $v$. Then note that $\{v_1,v_2,v_3\}\cap V(P_4)=\emptyset$ to avoid a forbidden subgraph. 
    Moreover, note that $E(\{ v_1, v_2, v_3 \}, P_4 ) = \emptyset$ and therefore $V(\textsf{paw})\cap P_4=\emptyset$. 
    Then, we have an induced $G_{6,7}$. 
\end{proof}

Therefore, we first assume that there is a vertex $v$ adjacent to $p_1$ and $p_2$. Then if there is another vertex $u\notin P_4$ adjacent to $v$, then $u$ is adjacent to $P_4$ to avoid a bull graph. Moreover, $\{p_1,p_2,p_3,p_4,v,u\}$ induces an $P_4^{(2,0,0,0)}$ or an $P_4^{(1,-1,0,0)}$. Moreover, there is no vertex adjacent to $u$ non-adjacent to $P_4$.
On the other hand, if there is another vertex $w$ adjacent to $P_4$ not adjacent to $v$, then $\{p_1,p_2,p_3,p_4,v,w\}$ induces $P_4^{(1,0,\pm 1,0)}$ or $P_4^{(1,0,0,\pm 1)}$ to avoid an induced forbidden subgraph. Moreover, there is no vertex $w_1$ adjacent to $w$ and non-adjacent to $P_4$. Therefore, $G$ is $P_4^{(n_1,-m_1,-m_2,n_2)}$ or $P_4^{(n_1,-m_1,n_2,-m_2)}$ for $n_i,m_i\geq 0$. By symmetry, a similar argument holds when $v$ is adjacent to $p_3$ and $p_4$.

Now, let us assume that there is a vertex $v$ adjacent to $p_1$, $p_2$ and $p_3$. Then if there is a vertex $u$ adjacent to $v$, it must be adjacent to other vertices in $P_4$ and $V(P_4)\cup\{v,u\}$ induces a $P_4^{(-1,1,0,0)}$ or a $P_4^{(0,1,-1,0)}$. Moreover, there is no vertex adjacent to $u$ non-adjacent to $P_4$. On the other hand, if there is a vertex $w$ not adjacent to $v$, adjacent to $P_4$, then $V(P_4)\cup \{v,w\}$ induces a $P_4^{(0,1,0,1)}$. In particular, $P_4^{(0,1,0,-1)}$ is ruled out by $G_{6,9}$.
Moreover, there is no vertex adjacent to $w$ non-adjacent to $P_4$. Therefore, $G$ is $P_4^{(-m_1,n_1,-m_2,n_2)}$ with $m_i,n_i\geq 0$. 

Third, for diameter 4, take a maximal path $P_5$ with vertices $\{p_1,p_2,p_3,p_4,p_5\}$, then any vertex $v$ adjacent to $P_5$ is adjacent to $\{p_2\}$, $\{p_3\}$, $\{p_4\}$, $\{p_i,p_{i+2}\}$ for $i=1,2,3$, $\{p_1,p_2,p_3\}$, $\{p_3,p_4,p_5\}$, or $\{p_1,p_3,p_5\}$.

\begin{claim}\label{claim3}
    There must be a vertex adjacent to $\{p_1,p_2,p_3\}$ or $\{p_3,p_4,p_5\}$.  
\end{claim}
\begin{proof}
    Assume there is no such vertex. Then, considering the induced \textsf{paw} with vertices $\{v_1,v_2,v_3,v_4\}$, notice that $V(P_5)\cap \{v_1,v_2,v_3\}=\emptyset$ to avoid an induced $G_{6,7}$, $G_{6,5}$, a house graph, or an induced $G_{6,12}$. Similarly, $v_4\notin P_5$. Moreover, similarly to the previous claim, we can show that there must be a forbidden induce subgraph in the subgraph induce by $V(\textsf{paw})\cap V(P_5)$.
\end{proof}

By Claim \ref{claim3}, assume that there is a vertex $v$ adjacent to $\{p_1,p_2,p_3\}$. Then, if there is a vertex $u$ adjacent to $v$, then $V(P_5)\cup\{u,v\}$ induces $P_5^{(-1,1,0,0,0)}$, or $P_5^{(0,1,-1,0,0)}$, or $P_5^{(0,2,0,0,0)}$. Moreover, there is no vertex adjacent to $u$ non-adjacent to $P_5$. On the other hand, if there is a vertex $w$ non-adjacent to $v$, then $V(P_5)\cup\{w,v\}$ induces a $P_5^{(0,1,0,0,-1)}$ or a $P_5^{(0,1,0,1,0)}$. Furthermore, there is no vertex adjacent to $w$ non-adjacent to $P_5$. Thus, $G$ must a $P_5^{(-m_1,n_2,-m_2,n_2,-m_3)}$ with $n_i,m_j\geq 0$ and $m_2\geq 1$ since there is an induced $C_4$. 
and this concludes the proof.

\end{proof}

\section{Graphs with at most 2 trivial distance ideals}\label{section:characterization}

In this section, we will complete the characterization of $\Lambda_2^\mathbb{Z}$, the graphs with at most 2 trivial distance ideals over $\mathbb{Z}[X]$. 
In order to complete this characterization, it remains to prove that the graphs outlined in Theorem~\ref{theo:main theorem2} possess a nontrivial third distance ideal. 
This aspect will be rigorously established through this section yielding Theorem~\ref{theo:main theorem}.

The main tool of this section will be Lemma~\ref{lem:deltaideal} that states that if $\Delta_i(D_X(G)|_{X={\bf d}})\neq 1$ for some ${\bf d}\in \mathbb{Z}^{n}$, then the distance ideal $I_i^\mathbb{Z}(G)$ is not trivial.
On the other hand, the next lemma implies that if $M$ and $N$ are two equivalent integer matrices, then $\Delta_i(M)=\Delta_i(N)$ for $i\geq0$.

\begin{lemma}\label{lem:norcot}\cite[Theorem 3]{Nbook}
    If $ M$ and $ N$ are two equivalent matrices with entries in a commutative ring, then, for each $i\geq0$, the determinantal ideals generated by the $i$-minors of $ M$ and $ N$ coincide, that is, for each $i\geq0$, the ideals generated by $minors_i({ M})$ and $minors_i({ N})$ are equal.
\end{lemma}

First, we take a closer look to bipartite graphs. 
In order to prove that every bipartite graph belongs to the family of graphs with at most two distance ideals, we first demonstrate that every $3\times 3$ minor in the distance matrix is an even number. 
Since bipartite graphs may have an arbitrary diameter, we will look at their distance matrix modulo 2. 
Thus, by Lemma~\ref{lem:norcot}, the parity of g.c.d. of the $k$-minors in the reduced matrix is the same as the parity of the g.c.d of the $k$-minors in the initial matrix.
Then, we will be able to prove that $\Delta_{3}(D(G))=2$ for every bipartite graph $G$ with the exception of $K_{2,2}$. 

\begin{lemma}\label{lemma:3minors}
    Every $3$-minor of the distance matrix of a connected bipartite graph is an even number.
\end{lemma}
\begin{proof}
Let $G$ be a bipartite graph with $a+b$ vertices where $a$ and $b$ are the orders of the partitions. 
Observe that the distance between two vertices within the same partition is even. 
In contrast, the distance between vertices in different partitions is odd. 
Then the distance matrix of $G$ with entries taken modulo 2 has the following form:
\[ 
{\sf L} := \begin{bmatrix}
    {\sf 0}_{a,a} & {\sf J}_{a,b} \\
    {\sf J}_{b,a} & {\sf 0}_{b,b}
\end{bmatrix},
\]
where ${\sf 0}_{n,m}$ and ${\sf J}_{n,m}$ denote an $n\times m$ all-zeros matrix and all-ones matrix, respectively. 
From matrix {\sf L}, it is easy to determine the parity of $3$-minors of the distance matrix $D(G)$ of $G$. 
To matrix ${\sf L}$, apply the following elementary operations. 
For each $i\in\lbrace 2,\dots, a\rbrace$, subtract first row to $i${\it-th} row. 
Similarly, subtract $(a+1)${\it-th} row to $(a+i')${\it-th} row for each $i'=2,\dots,b$. 
Subsequently, subtract the first column from each subsequent column associated with vertices in the same partition, and then subtract the column $a+1$ from all columns $a+j'$ for $j'=2,\dots, b$. 
Finally, interchange column $a+1$ with column $2$, and row $a+1$ with the second row. 
The obtained matrix is equivalent to 
$$
\begin{bmatrix}
    0&1&0&\cdots&0\\
    1&0&0&\cdots&0\\
    0&0&0&\cdots&0\\
    \vdots&\vdots&\vdots&\ddots&\vdots\\
    0&0&0&\cdots&0\\
\end{bmatrix}
$$ 
Therefore, $\text{det}({\sf L})=0$. 
\end{proof}

\begin{theorem}\label{theo:3thideals}
    Let $G$ be a connected bipartite graph with at least 4 vertices, with the exception of $K_{2,2}$, then $\Delta_3(D(G))$, the g.c.d. of the $3$-minors of $D(G)$, is equal to 2. 
\end{theorem}
\begin{proof} 
By Lemma~\ref{lemma:3minors}, it is sufficient to prove that the least positive number appearing as $3$-minor is $2$. 
Let $H$ be any induced connected bipartite subgraph of $G$ with $4$ vertices. 
Then $H$ is either $K_{1,3}$, $P_{4}$ or $K_{2,2}$. 
Bipatiteness of $G$ implies that distance between any pair of vertices in $H$ cannot be reduced in $G$. 
The distance matrix of each graph, $K_{1,3}$ and $P_{4}$, contains the following $3\times 3$ submatrices with determinant $2$:
\begin{equation*}
    \begin{array}{cc}
        \begin{bmatrix}
        0&2&1\\
        2&2&1\\
        1&1&0
    \end{bmatrix} &  \begin{bmatrix}
        0&1&1\\
        2&3&1\\
        1&2&0
    \end{bmatrix}\\
         K_{1,3}&P_{4} 
    \end{array}
\end{equation*}
Suppose $H=K_{2,2}$. 
Since $G\neq K_{2,2}$ and it is connected, there exists $w\in V(G\setminus H)$ adjacent to a vertex $x$ in $H$.
Suppose $x,y,z\in H$ such that $x,z$ are in the same partition, and $w$ is in the same partition as $y$.
If the edge $wz\notin E(G)$, we finish by considering the induced path obtained by $x,y,z,w$. 
The remaining case leads to the complete bipartite graph $K_{2,3}$, in which distances in $H$ cannot be reduced by $G$. 
A direct calculation proves that $2$ is a $3$-minor of $D(K_{2,3})$.  
\end{proof}

The remaining bipartite graph with four vertices, $K_{2,2}$, 
satisfies $\Delta_{3}(D(K_{2,2}))=4$.
And $\det(D(P_{3}))=4$. 

Therefore, $\Phi_{\mathbb{Z}}(G)=2$ for every bipartite graph $G$. 
In consequence, the family $\Lambda_{2}^\mathbb{Z}$ of graphs with at most two trivial distance ideals over $\mathbb{Z}[X]$ contains the connected bipartite graphs. 
Thus
$${\sf trees}\subseteq\text{\sf bipartite-graphs}\subseteq \Lambda_{2}^\mathbb{Z}.$$

Since the second invariant factor of bipartite graphs equals $1$, from Theorem~\ref{theo:3thideals} and Proposition~\ref{prop:eval1}, we get the following result. 

\begin{corollary}
    The third invariant factor of the distance matrix for any connected bipartite graph is equal to 2, except by $K_{2,2}$ and $P_{3}$. 
    The third invariant factor of the distance matrix of $K_{2,2}$ and $P_{3}$ is $4$.
\end{corollary}

From which follows the next interesting observation.

\begin{corollary}
    The invariant factors greater than 1 of the distance matrix of any bipartite graph are even.
\end{corollary}

In particular, the determinant of the distance matrix of a connected bipartite graph is even.
Therefore, this suggests that it is possible to extend, to bipartite graphs, Graham-Pollak-Lovász celebrated formula $\det(D(T_{n+1}))=(-1)^nn2^{n-1}$, see \cite{GP}, and Hou-Woo result \cite{HW} stating that $\SNF(D(T_{n+1}))=\sf{I}_2\oplus 2\sf{I}_{n-2}\oplus (2n)$, for any tree $T_{n+1}$ with $n+1$ vertices. 

In the following, we will finish to prove that the graphs considered in Theorem~\ref{theo:main theorem} encompass the entire family $\Lambda_{2}^\mathbb{Z}$. 
The graphs considered in Theorem~\ref{theo:main theorem} are obtained by the process of ``blowing up'' vertices of a graph into a clique or an independent set. 
We will see that this operation just adds zeros to the Smith normal form of the base graph. 

Given two integer vectors ${\sf a}\in\mathbb{Z}^{r}$ and ${\sf b}\in\mathbb{Z}^{s}$, the vector ${\sf a}\oplus{\sf b}\in\mathbb{Z}^{r+s}$ is the result of the concatenation of the vector ${\sf b}$ after the vector ${\sf a}$. 
We will use the notation ${\sf a}_{r}$ to refer to the vector $(a,a,\dots, a)\in\mathbb{Z}^{r}$. 
For example, ${\sf 1}_{3}\oplus{\sf 4}_{5}=(1,1,1,4,4,4,4,4)$. 

\begin{proposition}\label{prop:complete tripartites SNF}

Let $p,q,r$ be three arbitrary integers. 
If ${\sf d}={\sf 2}_{p}\oplus{\sf 2}_{q}\oplus{\sf 2}_{r}$, then the SNF of $D_X(K_{p,q,r})|_{X=\sf d}$ is equal to $\diag(1,1,4,0,\dots,0)$.   
\end{proposition}

\begin{proof}
    The generalized distance matrix $D_X(K_{p,q,r})$ of $K_{p,q,r}$ is 
\begin{equation*}
\begin{bmatrix}
    {\sf 2J}-(2-x){\sf I}&{\sf J}&{\sf J}\\
    {\sf J}&{\sf 2J}-(2-x){\sf I}&{\sf J}\\
    {\sf J}&{\sf J}&{\sf 2J}-(2-x){\sf I}
\end{bmatrix}
\end{equation*} where {\sf I} is the identity matrix and {\sf J} is the all-ones matrix of suitable size. 
We can apply elementary operations to the matrix $D_X(K_{p,q,r})|_{X=\sf d}$ as follows. For $i\in\lbrace 2,\dots, p\rbrace$, subtract row 1 from row $i$. Similarly, for $j\in\lbrace 2,\dots, q\rbrace$, subtract the row $p+1$ from the row $p+j$ and, finally, subtract the row $p+q+1$ from each row $p+q+k$ for each $k=2,\dots, r$. After performing these operations, continue subtracting column 1 from column $i$ for each $i=2,\dots, p$, column $p+1$ all columns $p+j$, $j=2,\dots, q$ and subtract column $p+q+1$ to columns $p+q+k$ for $k=2,\dots, r$. Finally, interchange column $p+1$ with the second column and rise row $p+1$ to the second row. Similarly, take the row $p+q+1$ to the third row and the column $p+q+1$ to the third row. 
After these finite elementary operations, we obtain
\begin{equation*}
D_X(K_{p,q,r})|_{X=\sf d}\sim\begin{bmatrix}
    2&1&1&0&\cdots&0\\
    1&2&1&0&\cdots&0\\
    1&1&2&0&\cdots&0\\
    0&0&0&0&\cdots&0\\
    \vdots&\vdots&\vdots&\vdots&\ddots&\vdots\\
    0&0&0&0&\cdots&0\\
\end{bmatrix}.
\end{equation*} 
Notice that the block 
\begin{equation*}
    \begin{bmatrix}
    2&1&1\\
    1&2&1\\
    1&1&2\\
    \end{bmatrix}
\end{equation*} is just $D_X(C_{3})|_{X=\sf u}$, where ${\sf u}=(2,2,2)$. 
Essentially, a complete tripartite graph is $C_{3}$ after the operation of blowing up each vertex in an independent set of vertices. 
It is easy to check that the SNF of $D_X(C_{3})|_{X=\sf u}$ is $\diag(1,1,4)$.
Therefore, 
\begin{equation*}
D_X(K_{p,q,r})|_{X=\sf d}\sim\begin{bmatrix}
    1&0&0&0&\cdots&0\\
    0&1&0&0&\cdots&0\\
    0&0&2&0&\cdots&0\\
    0&0&0&0&\cdots&0\\
    \vdots&\vdots&\vdots&\vdots&\ddots&\vdots\\
    0&0&0&0&\cdots&0\\
\end{bmatrix}.
\end{equation*}
\end{proof}

As a consequence, we have the following.

\begin{corollary}
    Complete tripartite graphs are contained in $\Lambda_{2}^\mathbb{Z}$. 
\end{corollary}

\begin{proposition}
    Let ${\sf d}={\sf 2}_{n-p}\oplus{\sf 1}_{p}$ for positive integers $n,p$ with $n>p$. 
    Then, $\SNF(D_X(K_{n-p,1,\dots,1})|_{X=\sf d})=\diag(1,1,0,\dots,0)$, and the complete multipartite graph $K_{n-p,1,1,\dots,1}$ with $n$ vertices belongs to $\Lambda_{2}^\mathbb{Z}$
\end{proposition}

\begin{proof}
    The evaluation of the generalized distance matrix of $K_{n-p,1,\dots,1}$ at $X$ equal to the vector ${\sf d}={\sf 2}_{n-p}\oplus{\sf 1}_{p}$ is
    $$\begin{bmatrix}
        2{\sf J}_{n-p,n-p}&{\sf J}_{n-p,p}\\
        {\sf J}_{p,n-p}&{\sf J}_{p,p}
    \end{bmatrix}.$$ 
    It is easy to see that $D_X(K_{n-p,1,\dots,1})|_{X=\sf d}$ is equivalent to the following matrix 
    $$\begin{bmatrix}
        2&1&0&\cdots&0\\
        1&1&0&\cdots&0\\
        0&0&0&\cdots&0\\
        \vdots&\vdots&\vdots&\ddots&\vdots\\
        0&0&0&\cdots&0\\
    \end{bmatrix}.$$ 
    Since $$\SNF\begin{bmatrix}
        2&1\\
        1&1\\
    \end{bmatrix}=\diag(1,1)$$ then, $\SNF(D_X(K_{n-p,1,\dots,1})|_{X=\sf d})=\diag(1,1,0,\dots,0).$ 
    By Lemma~\ref{lem:deltaideal}, $K_{n-p,1,1,\dots,1}$ belongs to $\Lambda_{2}^\mathbb{Z}$.
\end{proof}

Let $\Psi:=P_5^{(-m_{1},n_{1},-m_{2},n_{2},-m_{3})}$ be the graph described by the following diagram
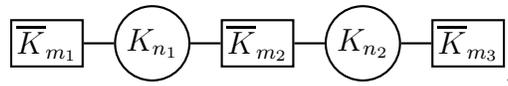
\begin{center}
\begin{tikzpicture}[scale=1.4,thick]
\tikzstyle{every node}=[minimum width=0pt, inner sep=2pt, circle]
\draw (-2,0) node (v1) [draw, rectangle] {$\overline K_{m_1}$};
\draw (-1,0) node (v2) [draw] {$K_{n_1}$};
\draw (0,0) node (v3) [draw, rectangle] {$\overline K_{m_2}$};
\draw (1,0) node (v4) [draw] {$K_{n_2}$};
\draw (2,0) node (v5) [draw,rectangle] {$\overline K_{m_3}$};
\draw (v1) -- (v2) -- (v3) -- (v4) -- (v5);
\end{tikzpicture}, 
\end{center} 
where $\overline K_{m}$ and $K_{n}$ denotes an independent set of size $m$ and a complete graph with $n$ vertices, respectively. 
Any edge in the diagram between $\overline K_{m}$ and $K_{n}$ means that each vertex in an edge $\overline K_{m}$ is adjacent to each vertex in $K_{n}$. 

\begin{proposition}
    Suppose $m_{1}, n_{1}, m_{2}, n_{2}, m_{3}>0$. 
    Let ${\sf d}={\sf 2}_{m_{1}}\oplus{\sf 1}_{n_{1}}\oplus{\sf 2}_{m_{2}}\oplus{\sf 1}_{n_{2}}\oplus{\sf 2}_{m_{3}}$.
    Then $\SNF(D_X(\Psi)|_{X=\sf d})=\diag(1,1,2,2,0,0,\dots,0)$ and $\Psi\in\Lambda_{2}^{\mathbb{Z}}$.
\end{proposition}

\begin{proof}
The generalized distance matrix of $\Psi$ evaluated at $X={\sf d}$ is
\begin{equation*}
D_X(\Psi)|_{X=\sf d}=\begin{bmatrix}
2{\sf J}&{\sf J}&2{\sf J}&3{\sf J}&4{\sf J}\\
{\sf J}&{\sf J}&{\sf J}&2{\sf J}&3{\sf J}\\
2{\sf J}&{\sf J}&2{\sf J}&{\sf J}&2{\sf J}\\
3{\sf J}&2{\sf J}&{\sf J}&{\sf J}&{\sf J}\\
4{\sf J}&3{\sf J}&2{\sf J}&{\sf J}&2{\sf J}
\end{bmatrix}.
\end{equation*} 
Therefore, after applying elementary operations, we obtain that that 
$$D_X(\Psi)|_{X=\sf d}\sim\begin{bmatrix}
    2&1&2&3&4&0&\cdots&0\\
    1&1&1&2&3&0&\cdots&0\\
    2&1&2&1&2&0&\cdots&0\\
    3&2&1&1&1&0&\cdots&0\\
    4&3&2&1&2&0&\cdots&0\\
    0&0&0&0&0&0&\cdots&0\\
    \vdots&\vdots&\vdots&\vdots&\vdots&\vdots&\ddots&\vdots\\
    0&0&0&0&0&0&\cdots&0\\
\end{bmatrix}.$$
The submatrix of left-hand side matrix, formed by taking the first 5 rows and 5 columns, is precisely the evaluation of the generalized distance matrix of $P_{5}$ evaluated at $X={\sf u}$, where $u=(2,1,2,1,2)$. 
A direct computation shows that the $\SNF$ of $D_X(P_{5})|_{X=\sf u}$ is $\diag(1,1,2,2,0)$. 
Therefore, $\SNF(D_X(\Psi)|_{X=\sf d})=\diag(1,1,2,2,0,0,\dots,0)$.
From which also follows that $\Psi\in\Lambda_{2}^{\mathbb{Z}}$.
\end{proof}

\begin{corollary}
    If $H$ is a connected induced subgraph of $\Psi$, then $H\in\Lambda_{2}^{\mathbb{Z}}$.
\end{corollary}
\begin{proof}
    Since the graph $\Psi$ is distance-hereditary, then the result follows from Lemma~\ref{lemma:inducemonotone}.
\end{proof}

Finally, consider the graph $\Omega:=P_4^{(n_{1},-m_{1},-m_{2},n_{2})}$ described by the diagram
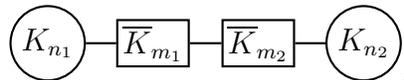
\begin{center}
\begin{tikzpicture}[scale=1.4,thick]
\tikzstyle{every node}=[minimum width=0pt, inner sep=2pt, circle]
\draw (-1,0) node (v2) [draw] {$K_{n_1}$};
\draw (0,0) node (v3) [draw, rectangle] {$\overline K_{m_1}$};
\draw (1,0) node (v4) [draw,rectangle] {$\overline K_{m_2}$};
\draw (2,0) node (v5) [draw] {$K_{n_2}$};
\draw (v2) -- (v3) -- (v4) -- (v5);
\end{tikzpicture}.
\end{center}

\begin{proposition}
    Let ${\sf d}={\sf 1}_{n_{1}}\oplus{\sf 2}_{m_{1}}\oplus{\sf 2}_{m_{2}}\oplus{\sf 1}_{n_{2}}$. 
    Then $\SNF(D_X(\Omega)|_{X=\sf d})=(1,1,3,3,0,0,\dots,0)$.
    Moreover, $\Omega\in\Lambda_{2}^\mathbb{Z}$.
\end{proposition}
\begin{proof}
The matrix
\begin{equation*}
    D_X(\Omega)|_{X=\sf d}=\begin{bmatrix}
        {\sf J}&{\sf J}&2{\sf J}&3{\sf J}\\
        {\sf J}&2{\sf J}&{\sf J}&2{\sf J}\\
        2{\sf J}&{\sf J}&2{\sf J}&{\sf J}\\
        3{\sf J}&2{\sf J}&{\sf J}&{\sf J}\\
    \end{bmatrix}
\end{equation*} is equivalent to the matrix
$$\begin{bmatrix}
    1&1&2&3&0&\cdots&0\\
    1&2&1&2&0&\cdots&0\\
    2&1&2&1&0&\cdots&0\\
    3&2&1&1&0&\cdots&0\\
    0&0&0&0&0&\cdots&0\\
    \vdots&\vdots&\vdots&\vdots&\vdots&\ddots&\vdots\\
    0&0&0&0&0&\cdots&0\\
\end{bmatrix},$$
which contains as submatrix $D_X(P_{4})|_{X=(1,2,2,1)}$. 
A straightforward computation shows that $\SNF(D_X(P_{4})|_{X=(1,2,2,1)})=\diag(1,1,3,3)$. 
Thus, $D_X(\Omega)|_{X=\sf d}$ is equivalent to $$\diag(1,1,3,3,0,\dots,0),$$ and $\Omega\in\Lambda_{2}^\mathbb{Z}$.
\end{proof}
\begin{corollary}
    If $H$ is a connected induced subgraph of $\Omega$, then $H\in\Lambda_{2}^\mathbb{Z}$.
\end{corollary}
\begin{proof}
    It follows since $\Omega$ is distance-hereditary.
\end{proof}

Finally, since any connected graph in $\Lambda_2^{\mathbb{Z}}$ is an $\{{\cal F}, \textsf{odd-holes}_7\}$-free graph and any such graph, all of them described in Theorem~\ref{theo:main theorem2}, has nontrivial third distance ideal as shown throughout this section. Then, our main result follows. 

\begin{theorem}\label{theo:main theorem}
    Let $G$ be a graph with $n$ vertices. Then $G \in \Lambda_2^{\mathbb{Z}}$ if and only if $G$ is one of the following:
    \begin{itemize}
        \item[i)] $C_5$;
        \item[ii)] a connected bipartite graph; 
        \item[iii)] a complete tripartite graph;
        \item[iv)] $K_{n-p+1,1,\ldots,1}$ where $p$ is the number of partitions;
        \item[v)] a connected induced subgraph of $P_5^{(-m_1,n_1,-m_2,n_2,-m_3)}$ with $m_1,n_1,m_2,n_2,m_3 \geq 0$;

        \item[vi)] or a connected induced subgraph of $P_4^{(n_1,-m_1,-m_2,n_2)}$ with $n_1,m_1,n_2,m_2 \geq 0$.
    \end{itemize}
\end{theorem}

In other words, we have shown that

\begin{theorem}\label{theo:main simple}
    Let $G$ be a connected graph with $n$ vertices. Then $G \in \Lambda_2^{\mathbb{Z}}$ if and only if $G$ is $\{{\cal F},\textsf{odd-hole}_7\}$-free.
\end{theorem}
    
\section{Graphs with at most two trivial distance ideals with coefficients over rational numbers}\label{sec:5}

Given a connected graph $G$, we have that $\Phi^{\mathbb{Z}} (G)\leq \Phi^{\mathbb{Q}}(G)$, therefore $\Lambda_{2}^{\mathbb{Q}}\subseteq \Lambda_{2}^{\mathbb{Z}}$.

In particular, notice that, with respect to ideals over the real numbers, $I_k^{\mathbb{R}} (G)=I_k^{\mathbb{Q}} (G)$. Moreover, we can extend this to any ring $\mathfrak{R}$ that contains the rational numbers. 

\begin{figure}[ht]
\centering
\begin{tabular}{c@{\extracolsep{1cm}}c@{\extracolsep{1cm}}c@{\extracolsep{1cm}}c}
 \begin{tikzpicture}[scale=1.4]\label{fig:forbidden graphs}
	\tikzstyle{every node}=[minimum width=0pt, inner sep=2pt, circle]
	\draw (-1,0) node (v1) [draw,rectangle] {\tiny $\overline K_2$};
	\draw (-.5,0) node (v2) [draw] {};
	\draw (-.25,0) node (v3) [draw] {};
	\draw (-0,0) node (v4) [draw] {};
	\draw (v1) -- (v2) -- (v3) -- (v4);
	\end{tikzpicture}    &
    \begin{tikzpicture}[scale=1.4]
 \tikzstyle{every node}=[minimum width=0pt, inner sep=2pt, circle]
	\draw (-1,0) node (v1) [draw] {};
	\draw (-.75,0) node (v2) [draw] {};
	\draw (-.50,0) node (v3) [draw] {};
	\draw (-0.25,0) node (v4) [draw] {};
        \draw (0,0) node (v5) [draw] {};
	\draw (v1) -- (v2) -- (v3) -- (v4) -- (v5);
	\end{tikzpicture}   &
    \begin{tikzpicture}[scale=1.4]
 \tikzstyle{every node}=[minimum width=0pt, inner sep=2pt, circle]
	\draw (-1,0) node (v1) [draw] {};
	\draw (-.5,0) node (v2) [draw,rectangle] {\tiny $\overline K_2$};
	\draw (0,0) node (v3) [draw] {};
	\draw (0.25,0) node (v4) [draw] {};
	\draw (v1) -- (v2) -- (v3) -- (v4);
	\end{tikzpicture} 
 &\begin{tikzpicture}[scale=1.4]
 \tikzstyle{every node}=[minimum width=0pt, inner sep=2pt, circle]
	\draw (-30:0.2) node (v1) [draw] {};
	\draw (30:0.35) node (v2) [draw] {};
	\draw (90:0.35) node (v3) [draw] {};
	\draw (150:0.35) node (v4) [draw] {};
        \draw (210:0.2) node (v5) [draw] {};
	\draw (v1) -- (v2) -- (v3) -- (v4) -- (v5) -- (v1);
	\end{tikzpicture}\\
   $H_1=P_4^{(-1,0,0,0)}$  & $P_5$ & $H_2=P_4^{(0,-1,0,0)}$& $C_5$\\
   &&&\\
    \begin{tikzpicture}[scale=1.4]
	\tikzstyle{every node}=[minimum width=0pt, inner sep=2pt, circle]
	\draw (-1,0) node (v2) [draw] {\tiny $K_{2}$};
	\draw (0,0) node (v3) [draw, rectangle] {\tiny $\overline K_{2}$};
	\draw (1,0) node (v4) [draw,rectangle] {\tiny $\overline K_{2}$};
	\draw (v2) -- (v3) -- (v4);
	\end{tikzpicture}&\begin{tikzpicture}[scale=1.4]
	\tikzstyle{every node}=[minimum width=0pt, inner sep=2pt, circle]
	\draw (-1,0) node (v2) [draw] {\tiny $K_{2}$};
	\draw (0,0) node (v3) [draw, rectangle] {\tiny $\overline K_{2}$};
	\draw (1,0) node (v4) [draw] {\tiny $K_{2}$};
	\draw (v2) -- (v3) -- (v4);
	\end{tikzpicture}
 &\begin{tikzpicture}[scale=1.4]
	\tikzstyle{every node}=[minimum width=0pt, inner sep=2pt, circle]
	\draw (-1,-1) node (v2) [draw, rectangle] {\tiny $\overline K_{2}$};
	\draw (0,0) node (v3) [draw, rectangle] {\tiny $\overline K_{2}$};
	\draw (1,-1) node (v4) [draw,rectangle] {\tiny $\overline K_{2}$};
	\draw (v2) -- (v3) -- (v4) -- (v2);
	\end{tikzpicture}
 &\\
       $H_3=P_3^{(1,-1,-1)}$ &$H_4=P_3^{(1,-1,1)}$&$K_{2,2,2}$&\\
\end{tabular}

 \caption{Forbidden induce subgraphs for $\Lambda_{2}^{\mathbb{Q}}$}
 \label{fig:reals}
\end{figure}
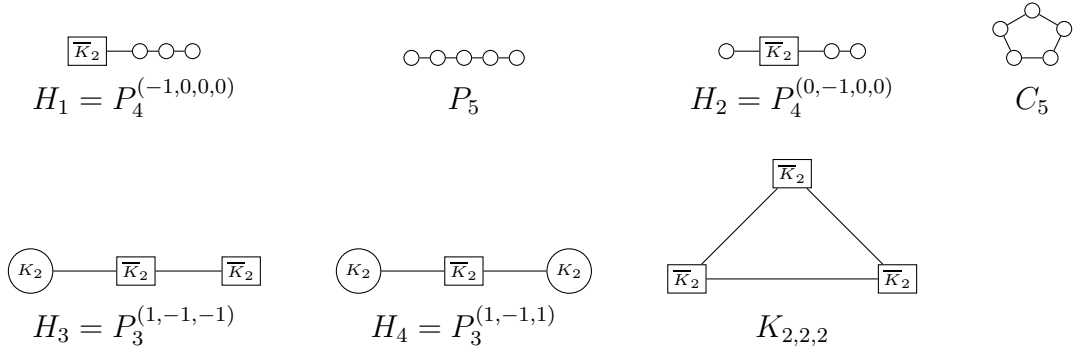
\begin{proposition}\label{prop:P5 is forb}
    The path graph $P_{5}$ is forbidden for $\Lambda_{2}^{\mathbb{Q}}$.
\end{proposition}
\begin{proof}
    Let $G$ be a graph with $P_{5}$ as induced graph. We aim to prove that $I_{3}^{\mathbb{Q}}(G)=\lrangle{1}$. 
    Let $u_{0},u_{1},u_{2},u_{3},u_{4}$ be the vertices of $P_{5}$ labeled such that $u_{i}$ is adjacent to $u_{i+1}$. 
    An easy computation shows that $I_{3}^{\mathbb{Q}}(P_{5})=\lrangle{1}$, thus if $d_{P_{5}}(u_{i},u_{j})=d_{G}(u_{i},u_{j})$ for all $i,j=0,\dots, 4$, then it follows that $I_{3}^{\mathbb{Q}}(G)$ is trivial. 
    Since $P_{5}$ is an induced subgraph, distances in $P_{5}$ different from $1$ can not be equal $1$ in $G$. That is, if $uv\not\in E(P_{5})$ then $uv\not\in E(G)$. Therefore, the distances in $P_{5}$ equal to $2$ are preserved within $G$. The only distances that can be reduced from $P_{5}$ are those distances that are at least $3$. This yields several cases.
    \begin{enumerate}
        \item $d_{G}(u_{0},u_{4})=2$.
        \item $d_{G}(u_{0},u_{4})=2$ and $d_{G}(u_{0},u_{3})=2$.
        \item $d_{G}(u_{0},u_{4})=2$, $d_{G}(u_{0},u_{3})=2$ and $d_{G}(u_{1},u_{4})=2$.
        \item $d_{G}(u_{0},u_{4})=2$ and $d_{G}(u_{1},u_{4})=2$.
        \item Cases \textit{1,2,3,4} but $d_{G}(u_{0},u_{4})=3$.
    \end{enumerate}
    
    A computation for all eight cases is provided in Code~\ref{lst:P5} in Appendix~\ref{codes}.
\end{proof}

\begin{proposition}\label{prop:C5 is forb}
    The cycle graph $C_{5}$ is forbidden for $\Lambda_{2}^{\mathbb{Q}}$. 
\end{proposition}
\begin{proof}
    Since $\text{diam}(C_{5})=2$, then the result follows by checking that $I_{3}^{\mathbb{Q}}(C_{5})=\lrangle{1}$.
\end{proof}

\begin{lemma}\cite{H77}
    If $G\in \Lambda_2^{\mathbb{Q}}$, then $G$ is distance-hereditary. 
\end{lemma}
\begin{proof}
    Since $G$ is in $\Lambda_2^{\mathbb{Z}}$, then, it is {\sf house} and {\sf gem} free. 
    Moreover, since it is $P_5$ and $C_5$ free by Propositions \ref{prop:P5 is forb} and \ref{prop:C5 is forb}, then $G$ is {\sf domino} and {\sf hole} free. 
    Equivalently, $G$ is distance hereditary.
\end{proof}

The previous lemma implies that now we only have to check that the third rational distance ideals of $H_1$, $H_2$, $H_3$, $H_4$ and $K_{2,2,2}$ are trivial to know that they are forbidden for $\Lambda_2^{\mathbb{Q}}$.

\begin{proposition}
    The graphs $H_1$, $H_2$, $H_3$, $H_4$ and $K_{2,2,2}$ in Figure~\ref{fig:forbidden graphs} are all forbidden for $\Lambda_{2}^{\mathbb{Q}}$.
\end{proposition}
\begin{proof}
    A direct computation shows the claim. Please refer to Code~\ref{lst:H}. 
\end{proof}


Now we present the main result of this section.

\begin{theorem}
    Let $G\in \Lambda_{2}^{\mathbb{Q}}$. Then $G$ is one of the following:
    \begin{itemize}
        \item[i)] $K_{n,m}$,
        \item[ii)] A complete tripartite graph $K_{1,n,m}$,
        \item[iii)] $K_{n-(p-1),1,\ldots,1}$ where $p$ is the number of partitions,
        \item[iv)] A connected induced subgraph of any of the following
        \begin{center}
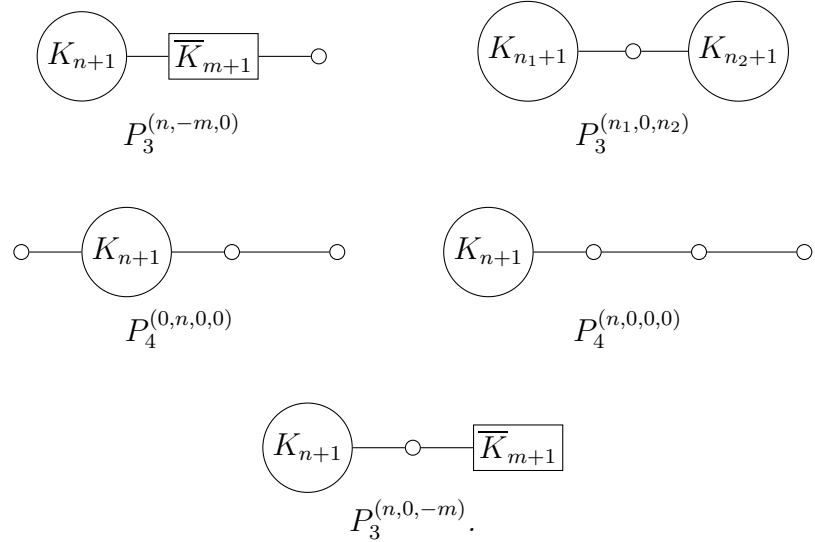
        \begin{tabular}{c@{\extracolsep{2cm}}c}
    
        \begin{tikzpicture}[scale=1.4]
	\tikzstyle{every node}=[minimum width=0pt, inner sep=2pt, circle]
	\draw (-1.25,0) node (v2) [draw] {$K_{n+1}$};
	\draw (0,0) node (v3) [draw, rectangle] {$\overline K_{m+1}$};
	\draw (1,0) node (v4) [draw] { };
	\draw (v2) -- (v3) -- (v4);
	\end{tikzpicture} 
        &
        \begin{tikzpicture}[scale=1.4]
	\tikzstyle{every node}=[minimum width=0pt, inner sep=2pt, circle]
	\draw (-1,0) node (v2) [draw] {$K_{n_1+1}$};
	\draw (0,0) node (v3) [draw] { };
	\draw (1,0) node (v4) [draw] {$K_{n_2+1}$};
	\draw (v2) -- (v3) -- (v4);
	\end{tikzpicture}\\
    $P_3^{(n,-m,0)}$ & $P_3^{(n_1,0,n_2)}$
\end{tabular}
        \end{center}
    
        \begin{center}
        \begin{tabular}{c@{\extracolsep{1.3cm}}c}

        \begin{tikzpicture}[scale=1.4]
	\tikzstyle{every node}=[minimum width=0pt, inner sep=2pt, circle]
	\draw (-2,0) node (v1) [draw] {};
	\draw (-1,0) node (v2) [draw] {$K_{n+1}$};
	\draw (0,0) node (v3) [draw] {};
	\draw (1,0) node (v4) [draw] {};
	\draw (v1) -- (v2) -- (v3) -- (v4);
	\end{tikzpicture} 
        &
        \begin{tikzpicture}[scale=1.4]
	\tikzstyle{every node}=[minimum width=0pt, inner sep=2pt, circle]
	\draw (-1,0) node (v2) [draw] {$K_{n+1}$};
	\draw (0,0) node (v3) [draw] {};
	\draw (1,0) node (v4) [draw] {};
        \draw (2,0) node (v5) [draw] {};
	\draw (v2) -- (v3) -- (v4) -- (v5);
	\end{tikzpicture}\\
        $P_4^{(0,n,0,0)}$ & $P_4^{(n,0,0,0)}$
        \end{tabular}
        \end{center}
        \begin{center}

        \begin{tikzpicture}[scale=1.4]
	\tikzstyle{every node}=[minimum width=0pt, inner sep=2pt, circle]
	\draw (-1,0) node (v2) [draw] {$K_{n+1}$};
	\draw (0,0) node (v3) [draw] { };
	\draw (1,0) node (v4) [draw,rectangle] {$\overline K_{m+1}$};
	\draw (v2) -- (v3) -- (v4);
	\end{tikzpicture}\\
        $P_3^{(n,0,-m)}$.
        \end{center}
        
    \end{itemize}
\end{theorem}
\begin{proof}
    Notice that the cases presented are the six cases in the characterization of $\Lambda_2^{\mathbb{Z}}$ restricted by the forbidden induced subgraphs in Figure~\ref{fig:reals}. 
    Now, it remains to show that the third rational distance ideals of these graphs are nontrivial.
    For this, first note that if the third distance ideal of a graph is trivial over $\mathbb{Q}[X]$, then any evaluation of $X=(x_1,\ldots , x_n)$ also yields the trivial ideal over $\mathbb{Q}[X]$. 
    Thus, we will approach the problem by showing that for each particular case, the third rational distance ideal over a specific evaluation of $X$ is nontrivial.

    For $G=K_{1,n,m}$, let $x,\ y$ and $z$ be three indeterminates and let
    \[x_i=\begin{cases}
        x & \text{ if }i=1,\\
        y & \text{ if }i\in\{2,\ldots ,n+1 \},\\
        z & \text{otherwise.}
    \end{cases}\] 
    We will show that the third rational distance ideal of $G$ is a nontrivial ideal over $\mathbb{Q}[x,y,z]$. Specifically, we claim that this ideal is given by $\lrangle{x-\frac{2}{3},y-2,z-2}$ whenever $n,m\geq 2$.
    Explicit computations for $n\leq 3$ and $m \leq 6$ are given in Code \ref{lst:caseii}.
    
    Furthermore, a basis for the ideal is found by going through every possible 3-minor (polynomial) considering all possible induced graphs $H$, of order $3\leq |H|\leq 6$.
    Thus, since no new structure appears for such induced subgraphs if $n> 3$ or $m> 6$, then the third rational distance ideal of $G$ is also $\lrangle{x-\frac{2}{3},y-2,z-2}$ in this case. 
    Moreover, $K_{n,m}$ is an induced subgraph of $K_{n,m}$, then its third distance ideal is also nontrivial.
    
    Now, for $G=P_3^{(-n,m,0)}$, let 
    \[x_i=\begin{cases}
        x & \text{ if }v_i\text{ is in the clique }K_{n+1},\\
        y & \text{ if }v_i\text{ is in the independent set }\overline K_{m+1},\\
        z & \text{otherwise.}
    \end{cases}\]  
    
    Then, we claim that the third rational distance ideal is given by $\lrangle{x-1,y-2,z-5}$ whenever $n,m\geq 1$. 
    An explicit computation for $1\leq n,m \leq 5$ is given in Code \ref{lst:caseiv1}.
    As in the previous case, since no new structure appears for such induced subgraphs if $n> 5$ or $m> 5$, then this result holds for any $n,m\geq 1$.  
    Also, note that case iii) is an induced subgraph of $P_3^{(-n,m,0)}$ (erasing the ``rightmost" vertex), then its third distance ideal is nontrivial.

    We can proceed similarly with the other graphs in iv). 
    The third rational distance ideal of $P_3^{(-n_1,0,=n_2)}$ is $\lrangle{x-1,y-\frac{2}{3},z-1}$ if $n_1,n_2\geq 1$, with
    \[x_i=\begin{cases}
        x & \text{ if }v_i\text{ is in the leftmost clique }K_{n_1+1},\\
        z & \text{ if }v_i\text{ is in the rightmost clique }K_{n_2+1},\\
        y & \text{otherwise.}
    \end{cases}\]  
    This can be verified with Code~\ref{lst:caseiv2}. 
    For $P_3^{(-n,0,m)}$, $n\geq 2$ and $m\geq 1$, its third rational distance ideal is $\lrangle{x-1,y-\frac{1}{2},z-2}$ with
    \[x_i=\begin{cases}
        x & \text{ if }v_i\text{ is in the clique }K_{n+1},\\
        z & \text{ if }v_i\text{ is in the independent set }\overline K_{m+1},\\
        y & \text{otherwise.}
    \end{cases}\]  
    Code~\ref{lst:caseiv3} can be used to check it.
    
    On the other hand, let
    \[x_i=\begin{cases}
        w & \text{ if }v_i\text{ is the leftmost vertex},\\
        x & \text{ if }v_i\text{ is in the clique }K_{n+1},\\
        y & \text{ if }v_i\text{ is in the ``first'' vertex to the right of }K_{n+1},\\
        z & \text{ if }v_i\text{ is the rightmost vertex},
    \end{cases}\]  
    then the third rational distance ideal of $P_4^{(0,-n,0,0)}$ is $\lrangle{w,x-1,y,z-3}$ for $n\geq 1$, see Code~\ref{lst:caseiv4}.
    Lastly, if
    \[x_i=\begin{cases}
        w & \text{ if }v_i\text{ is in the clique }K_{n+1},\\
        x & \text{ if }v_i\text{ is in the ``first'' vertex to the right of }K_{n+1},\\
        y & \text{ if }v_i\text{ is in the ``second'' vertex to the right of }K_{n+1},\\
        z & \text{ if }v_i\text{ is the rightmost vertex,}
    \end{cases}\]  
    then third rational distance ideal of $P_4^{(-n,0,0,0)}$ is $\lrangle{w-1,x-\frac{4}{5},y+1,z-4}$ for $n\geq 1$, see Code~\ref{lst:caseiv5}.

\end{proof}

The computation of the third rational distance ideal for $X=(x_1,\ldots ,x_n)$, in general, is more difficult.

\begin{proposition}
    The third distance ideal of $K_{n,m}$ over the rational numbers is not trivial.
\end{proposition}
\begin{proof}
    Assume that $n,m\geq 3$. We will compute explicitly a set of generators for $I_{3}^{\mathbb{Q}}(K_{n,m})$. For an arbitrary squared matrix, ${\sf A}$, we define ${\sf A}(\omega)$ to be the matrix obtained from ${\sf A}$ by replacing its diagonal with some vector of variables $\omega$. With this notation, the generalized distance matrix is  
    $${\sf D}_{X}(K_{m,n})=
    \begin{bmatrix}
    2{\sf J}({\sf x})_{n,n}&{\sf J}_{n,m}\\
    {\sf J}_{m,n}&2{\sf J}({\sf y})_{m,m}
    \end{bmatrix},$$ where ${\sf x}=(x_{1},\dots, x_{n})$ and ${\sf y}=(y_{1},\dots, y_{m})$. We may perform elementary operations to ${\sf D}_{X}(K_{m,n})$. In particular, subtracting the first row to the second, third and so on up to the $n-$th row and subtracting the last row to all the $n+1,\dots, n+m-1$ rows yields the equivalent matrix
    \begin{equation*}
    \begin{bmatrix}
        x_{1}&2&2&\cdots&2&1&1&\cdots&1\\
        2-x_{1}&x_{2}-2&0&\cdots&0&0&0&\cdots&0\\
        2-x_{1}&0&x_{3}-2&\cdots&0&0&0&\cdots&0\\
        \vdots&\vdots&\vdots&\ddots&\vdots&\vdots&\vdots&\ddots&\vdots\\
        2-x_{1}&0&0&\cdots&x_{n}-2&0&0&\cdots&0\\
        0&0&0&\cdots&0&y_{1}-2&0&\cdots&2-y_{m}\\
        0&0&0&\cdots&0&0&y_{2}-2&\cdots&2-y_{m}\\
        \vdots&\vdots&\vdots&\ddots&\vdots&\vdots&\vdots&\ddots&\vdots\\
        1&1&1&\cdots&1&2&2&\cdots&y_{m}\\
    \end{bmatrix}.
    \end{equation*} 
    We have just applied elementary operations over the ring $\mathbb{Z}$ of integers. After applying operations over $\mathbb{Q}$, it is easy to check that the above matrix is equivalent to 
    \begin{equation*}
    L:=\begin{bmatrix}
        \frac{1}{3}(2-x_{1})&0&0&\cdots&0&1&1&\cdots&\frac{1}{3}(2y_{m}-1)\\
        2-x_{1}&x_{2}-2&0&\cdots&0&0&0&\cdots&0\\
        2-x_{1}&0&x_{3}-2&\cdots&0&0&0&\cdots&0\\
        \vdots&\vdots&\vdots&\ddots&\vdots&\vdots&\vdots&\ddots&\vdots\\
        2-x_{1}&0&0&\cdots&x_{n}-2&0&0&\cdots&0\\
        0&0&0&\cdots&0&y_{1}-2&0&\cdots&2-y_{m}\\
        0&0&0&\cdots&0&0&y_{2}-2&\cdots&2-y_{m}\\
        \vdots&\vdots&\vdots&\ddots&\vdots&\vdots&\vdots&\ddots&\vdots\\
        \frac{1}{3}(x_{1}-1)&1&1&\cdots&1&0&0&\cdots&\frac{1}{3}(1-y_{m})\\
    \end{bmatrix}.
    \end{equation*}
    Let $I,J\subset [n+m]$ be two sets of 3 indices. 
    Let $L[I;J]$ be the submatrix of L obtained from $I$ and $J$. 
    Suppose first $\max{I},\max{J}\leq n$. 
    If $1\not\in I\cup J$, then $\det{L[I;J]}\neq 0$ if and only if $I=J$ and hence $\det{L[I;J]}=(x_{i}-2)(x_{j}-2)(x_{k}-2)$. 
    If $1\in I$ and $1\not\in J$, then $\det{L[I;J]}=0$ thus we must have $1\in I\cap J$. 
    Moreover, if $I\neq J$, then $L[I;J]$ contains an all-zero column. 
    The case $I=J=\lbrace 1,i,j\rbrace$ yields $\det{L[I;J]}=\frac{1}{3}(2-x_{1})(x_{i}-2)(x_{j}-2)$.
    Now assume $\max{I}\leq n< \max{J}$. 
    In this case, either one, two or all three indices of $J$ may be greater than $n$. 
    If two indices of $J$ are greater than $n$ and $1\in I\cap J$, then $L[I;J]$ contains a $2\times 2$ all-zero block implying that its determinant is $0$. 
    If $1\not\in J$, then $L[I;J]$ contains at least one all-zero row. 
    Suppose just one index of $J$ is greater than $n$. If $1\in I\cap J$, then $\det{L[I;J]}$ is either $-(2-x_{1})(x_{i}-2)$ or $-\frac{1}{3}(2y_{m}-1)(2-x_{1})(x_{i}-2)$ whether $\max {J}<n+m$ or $\max{J}=n+m$. 
    If $1\not\in J$, then all $L[I;J]$ contains an all-zero row and if $1\in J$ and $1\not\in I$, $L[I;J]$ contains an all-zero column.

    The matrix $L$ exhibits a symmetry among the four $2\times 2$ blocks of maximum size. Thus, we may conclude that the only nonzero $3-$minors of $L$ are $(y_{i}-2)(y_{j}-2), \frac{1}{3}(1-y_{m})(y_{i}-2)(y_{j}-2), -(2-y_{m})(y_{i}-2)$ and $-\frac{1}{3}(2-x_{1})(2-y_{m})(y_{i}-2)$.

    Finally, we may conclude that $I_{3}^{\mathbb{Q}}(K_{n,m})$ is generated by the set $$\lbrace x_{i}-2,y_{j}-2\mid i=1,\dots,n,~j=1,\dots,m\rbrace.$$ That is, $$I_{3}^{\mathbb{Q}}(K_{n,m})=\lrangle{x_{i}-2,y_{j}-2\mid i=1,\dots,n,~j=1,\dots,m},$$ which is not trivial.\\
    Now, if $n=2$ and $m\geq 3$, then $I_{3}^{\mathbb{Q}}(K_{n,m})=\lrangle{x_1x_2-\frac{1}{2}x_0-\frac{1}{2}x_1-2,y_{j}-2\mid j=1,\dots,m}$.
    On the other hand,
    $I_{3}^{\mathbb{Q}}(K_{2,2})=\langle x_1x_2 - \frac{1}{2}x_1 - \frac{1}{2}x_2 - 2, x_1x_3 - 2x_1 - 2x_3 + 4, x_2x_3 - 2x_2 - 2x_3 + 4, x_1x_4 - 2x_1 - 2x_4 + 4, x_2x_4 - 2x_2 - 2x_4 + 4, x_3x_4 - \frac{1}{2}x_3 - \frac{1}{2}x_4 - 2\rangle$, and $I_{3}^{\mathbb{Q}}(K_{1,m})$, for $m\geq 2$, is not trivial by \cite[Theorem 17]{at}.
\end{proof}



\section{Graphs with at most two trivial distance univariate ideals}\label{sec:6}

The univariate case is given by setting all indeterminates equal, say $x_i=t$ for any integer $i$. We denote the corresponding distance ideals by $I_k^{\mathfrak{R}} (G,t)=\langle \minors_k (t{\sf I}+D(G) ) \rangle$. 
In particular, note that $I_{|G|}^{\mathfrak{R}}(G,t)$ is generated by the characteristic polynomial of $-D(G)$. 
Let $\Phi^{\mathfrak{R}} (G,t)$ be the maximum integer $k$ such that $I_k^{\mathfrak{R}} (G,t)$ is trivial and let $\Lambda_2^{t,\mathfrak{R}}$ be the set of graphs with at most two trivial univariate distance ideals over $\mathfrak{R}$.

Similarly to the previous section, 
\[
\Phi^{\mathfrak{R}} (G) \leq \Phi^{\mathfrak{R}} (G,t).
\]

Thus,
\[
\Lambda_2^{t,\mathfrak{R}}\subset \Lambda_2^{\mathfrak{R}}.
\]

This scenario is much more restrictive since the {\sf paw} is a forbidden induced subgraph for the case $\mathfrak{R}=\mathbb{Z}$, as well as, $K_5-e=K_{2,1,1,1}$. 
Moreover, $P_4$, the {\sf diamond} graph and the {\sf paw} graph are forbidden for the case $\mathfrak{R}=\mathbb{Q}$.

Therefore, we have the following results.
\begin{theorem}
    Let $G$ be a connected graph. Then $G\in \Lambda_2^{t,\mathbb{Z}}$ if and only if is one the following
    \begin{itemize}
        \item[i)] $C_5$;
        \item[ii)] A connected bipartite graph, or 
        \item[iii)] $K_{n,m,r}$.
    \end{itemize}
\end{theorem}

\begin{proof}
These characterization results from restricting the one given for $\Lambda_2^{\mathbb{Z}}$ further by the paw graph and $K_5-e$. Thus, if $G\in \Lambda_2^{t,\mathbb{Z}}$, then $G$ falls into one of these 3 cases. Conversely, $I_3^{t,\mathbb{Z}} (C_5)=\lrangle{t + 6, 11}$. The third univariate distance ideal of a connected bipartite graph is nontrivial also by Theorem~\ref{theo:3thideals}. Moreover, $I_3^{t,\mathbb{Z}} (K_3)=\lrangle{t^3 - 3t + 2}$
and $I_3^{t,\mathbb{Z}} (K_{n,m,r})$ is nontrivial for $n,m,r\geq 1$ by Proposition~\ref{prop:complete tripartites SNF}.

\end{proof}

\begin{theorem}
    Let $G\in \Lambda_2^{t,\mathbb{Q}}$, then $G$ is a complete graph or a complete bipartite graph.
\end{theorem}

Note that, curiously, we have $\Lambda_2^{t,\mathbb{Q}}=\Lambda_1^{\mathbb{Z}}$.

\section{Conclusion}

Despite that distance ideals are not monotone under taking induced subgraphs, the characterizations of connected graphs with one trivial distance ideal over $\mathbb{Z}[X]$ and over $\mathbb{Q}[X]$ were obtained in \cite{at} in terms of induced subgraphs.
Later in \cite{alfaro2}, it was proven that the family $\Lambda_2^{\mathbb{Z}}$ of connected graphs with at most two trivial distance ideals over $\mathbb{Z}[X]$ were $\{{\cal F},\textsf{odd-holes}_{7}\}$-free. 
Therefore, since $\{{\cal F},\textsf{odd-holes}_{7}\}$ is an infinite set of minimal forbidden graphs, trying to obtain a characterization of $\Lambda_2^{\mathbb{Z}}$ was considered a difficult problem \cite{alfaro2}. 
Finally, in this paper, we give a characterization of $\Lambda_2^{\mathbb{Z}}$.
Also, we give the characterizations of $\Lambda_2^{\mathbb{Q}}$, $\Lambda_2^{t,\mathbb{Z}}$ and $\Lambda_2^{t,\mathbb{Q}}$.

In the scaffolding of the proof, we found that the determinant of the distance matrix of a connected bipartite graph is even, this suggests that it could be possible to extend, to connected bipartite graphs, the Graham-Pollak-Lovász celebrated formula $\det(D(T_{n+1}))=(-1)^nn2^{n-1}$, and the Hou-Woo result stating that $\SNF(D(T_{n+1}))=\sf{I}_2\oplus 2\sf{I}_{n-2}\oplus (2n)$, for any tree $T_{n+1}$ with $n+1$ vertices.
Also, it would be interesting to see until which $n$ the characterizations of $\Lambda_n^{\mathbb{Z}}$ will not be possible to obtain in terms of induced subgraphs and shed some light on the reasons why this happens.

\section*{Acknowledgement}
The authors were supported by SNI and CONACyT.
The research of Ralihe R. Villagran was supported, in part, by the National Science Foundation through the DMS Award $\#$1808376 which is gratefully acknowledged.


\appendix
\section{Codes}\label{codes}

In this appendix, we provide the codes written in SageMath~\cite{sage} that were used in the manuscript.

\begin{lstlisting}[language=Python, caption={$P_{5}$ is forbidden for $\Lambda_{2}^{\mathbb{Q}}$}, label={lst:P5}]
G=graphs.PathGraph(5)
R=PolynomialRing(QQ,G.order(),x)
D=G.distance_matrix()
DX=diagonal_matrix(R.gens())+D
g=ideal(DX.minors(3)).groebner_basis()
print(g==[1])
i=1
for d in (2,3):
    if d==3: i=5

    DX[0,4]=d
    g=ideal(DX.minors(3)).groebner_basis()
    print(i,g==[1])
    DX=diagonal_matrix(R.gens())+D

    DX[0,4]=d
    DX[0,3]=2
    g=ideal(DX.minors(3)).groebner_basis()
    print(i+1,g==[1])
    DX=diagonal_matrix(R.gens())+D

    DX[0,4]=d
    DX[0,3]=2
    DX[1,4]=2
    g=ideal(DX.minors(3)).groebner_basis()
    print(i+2,g==[1])
    DX=diagonal_matrix(R.gens())+D

    DX[0,4]=d
    DX[1,4]=2
    g=ideal(DX.minors(3)).groebner_basis()
    print(i+3,g==[1])
\end{lstlisting}

\begin{lstlisting}[language=Python, caption={$H_{1},H_{2},H_{3},H_{4}$ are forbidden for $\Lambda_{2}^{\mathbb{Q}}$}, label={lst:H}]
H1=Graph({0:[1],1:[2],2:[3],3:[4],4:[1]})
R=PolynomialRing(QQ,H1.order(),x)
D=H1.distance_matrix()
DX=diagonal_matrix(R.gens())+D
g1=ideal(DX.minors(3)).groebner_basis()

H2=Graph({0:[1],1:[2],2:[3],4:[0,2]})
R=PolynomialRing(QQ,H2.order(),x)
D=H2.distance_matrix()
DX=diagonal_matrix(R.gens())+D
g2=ideal(DX.minors(3)).groebner_basis()

H3=Graph({0:[1,2,3],1:[2,3],2:[4,5],3:[4,5]})
R=PolynomialRing(QQ,H3.order(),x)
D=H3.distance_matrix()
DX=diagonal_matrix(R.gens())+D
g3=ideal(DX.minors(3)).groebner_basis()

H4=Graph({0:[1,2,3],1:[2,3],2:[4,5],3:[4,5],4:[5]})
R=PolynomialRing(QQ,H4.order(),x)
D=H4.distance_matrix()
DX=diagonal_matrix(R.gens())+D
g4=ideal(DX.minors(3)).groebner_basis()

K222=graphs.CompleteMultipartiteGraph([2,2,2])
R=PolynomialRing(QQ,K222.order(),x)
D=H4.distance_matrix()
DX=diagonal_matrix(R.gens())+D
g=ideal(DX.minors(3)).groebner_basis()

g1==[1] and g2==[1] and g3==[1] and g4==[1] and g==[1]
\end{lstlisting}

\begin{lstlisting}[language=Python, caption=Third distance ideal for $K_{1,n,m}$., label={lst:caseii}]
R = PolynomialRing(QQ,'w,x,y,z')
R.inject_variables()
def third_dist_ideal(g,X):
    D = g.distance_matrix()
    M = diagonal_matrix(X)+D
    I = R.ideal(M.minors(3))
    return list(I.groebner_basis())

def Equis(V):
    X =[]
    for i in V:
        if i == 0:
            X.append(x)
        elif i < 4:
            X.append(y)
        else:
            X.append(z)
    return X

G = graphs.CompleteMultipartiteGraph([1,3,6])
zs = [4,5,6,7,8,9]
ys = [1,2,3]

for i in range(2,7):
    for j in range(1,4):
        H = G.subgraph(ys[:j]+zs[:i])
        print(third_dist_ideal(H,Equis(H.vertices(sort=True))))
print('next')
for i in range(2,7):
    for j in range(1,4):
        H = G.subgraph([0]+ys[:j]+zs[:i])
        print(third_dist_ideal(H,Equis(H.vertices(sort=True))))
\end{lstlisting}

\begin{lstlisting}[language=Python, caption=Third distance ideal for $P_3^{(-n,m,0)}$., label={lst:caseiv1}]
R = PolynomialRing(QQ,'x,y,z')
R.inject_variables()
def third_dist_ideal(g,X):
    D = g.distance_matrix()
    M = diagonal_matrix(X)+D
    I = R.ideal(M.minors(3))
    return list(I.groebner_basis())

def Equis(V):
    X =[]
    for i in V:
        if i < 6:
            X.append(x)
        elif i < 12:
            X.append(y)
        elif i == 12:
            X.append(z)
    return X

G = Graph({0:[3,4,2,1,5,6,9,7,10,11,8], 1:[5,4,2,3,6,9,7,10,11,8], 2:[3,4,5,6,9,7,10,11,8], 3:[4,5,6,9,7,10,11,8], 4:[5,6,9,7,10,11,8], 5:[6,9,7,10,11,8], 6:[12], 7:[12], 8:[12], 9:[12], 10:[12], 11:[12], 12:[]})
xs = [0,1,2,3,4,5]
ys = [6,7,8,9,10,11]
    
for i in range(2,7):#this is case iii)
    for j in range(2,7):
        H = G.subgraph(xs[:i]+ys[:j])
        print(third_dist_ideal(H,Equis(H.vertices(sort=True))))
 
for i in range(7):#general case for the first graph in case iv)
    for j in range(1,7):
        H = G.subgraph(xs[:i]+ys[:j]+[12])
        if H.order()>2:
            print(third_dist_ideal(H,Equis(H.vertices(sort=True))))
\end{lstlisting}

\begin{lstlisting}[language=Python, caption=Third distance ideal for $P_3^{(-n_1,0,-n_2)}$., label={lst:caseiv2}]
R = PolynomialRing(QQ,'x,y,z')
R.inject_variables()
def third_dist_ideal(g,X):
    D = g.distance_matrix()
    M = diagonal_matrix(X)+D
    I = R.ideal(M.minors(3))
    return list(I.groebner_basis())

def Equis(V):
    X =[]
    for i in V:
        if i < 6:
            X.append(x)
        elif i == 6:
            X.append(y)
        else:
            X.append(z)
    return X

G = Graph({0:[1,2,3,4,5,6], 1:[2,3,4,5,6], 2:[3,4,5,6], 3:[4,5,6], 4:[5,6], 5:[6], 6:[7,8,9,10,12,11], 7:[8,9,10,11,12], 8:[9,10,11,12], 9:[10,11,12], 10:[11,12], 11:[12], 12:[]})
xs = [0,1,2,3,4,5]
ys = [7,8,9,10,11,12]
    
for i in range(1,7):
    for j in range(i,7):#by symmetry
        H = G.subgraph(xs[:i]+[6]+ys[:j])
        print(third_dist_ideal(H,Equis(H.vertices(sort=True))))
\end{lstlisting}

\begin{lstlisting}[language=Python, caption=Third distance ideal for $P_3^{(-n,0,m)}$., label={lst:caseiv3}]
R = PolynomialRing(QQ,'x,y,z')
R.inject_variables()
def third_dist_ideal(g,X):
    D = g.distance_matrix()
    M = diagonal_matrix(X)+D
    I = R.ideal(M.minors(3))
    return list(I.groebner_basis())

def Equis(V):
    X =[]
    for i in V:
        if i < 6:
            X.append(x)
        elif i == 6:
            X.append(y)
        else:
            X.append(z)
    return X

G = Graph({0:[1,2,3,4,5,6], 1:[2,3,4,5,6], 2:[3,4,5,6], 3:[4,5,6], 4:[5,6], 5:[6], 6:[12,7,8,11,10,9], 7:[], 8:[], 9:[], 10:[], 11:[], 12:[]})
xs = [0,1,2,3,4,5]
ys = [7,8,9,10,11,12]
    
for i in range(1,7):
    for j in range(1,7):
        H = G.subgraph(xs[:i]+[6]+ys[:j])
        print(third_dist_ideal(H,Equis(H.vertices(sort=True))))
\end{lstlisting}

\begin{lstlisting}[language=Python, caption=Third distance ideal for $P_4^{(0,-n,0,0)}$., label={lst:caseiv4}]
R = PolynomialRing(QQ,'w,x,y,z')
R.inject_variables()
def third_dist_ideal(g,X):
    D = g.distance_matrix()
    M = diagonal_matrix(X)+D
    I = R.ideal(M.minors(3))
    return list(I.groebner_basis())

def Equis(V):
    X =[]
    for i in V:
        if i == 0:
            X.append(w)
        elif i < 7:
            X.append(x)
        elif i==7:
            X.append(y)
        else:
            X.append(z)
    return X

G = Graph({0:[1,6,2,4,3,5], 1:[2,3,4,5,6,7], 2:[3,4,5,6,7], 3:[4,5,6,7], 4:[5,6,7], 5:[6,7], 6:[7], 7:[8], 8:[]})
xs = [1,2,3,4,5,6]

for i in range(1,7):
    H = G.subgraph([0]+xs[:i]+[7,8])
    print(third_dist_ideal(H,Equis(H.vertices(sort=True))))
\end{lstlisting}

\begin{lstlisting}[language=Python, caption=Third distance ideal for $P_4^{(-n,0,0,0)}$., label={lst:caseiv5}]
R = PolynomialRing(QQ,'w,x,y,z')
R.inject_variables()
def third_dist_ideal(g,X):
    D = g.distance_matrix()
    M = diagonal_matrix(X)+D
    I = R.ideal(M.minors(3))
    return list(I.groebner_basis())

def Equis(V):
    X =[]
    for i in V:
        if i < 6:
            X.append(w)
        elif i == 6:
            X.append(x)
        elif i==7:
            X.append(y)
        else:
            X.append(z)
    return X

G = Graph({0:[1,2,3,4,5,6], 1:[2,3,4,5,6], 2:[3,4,5,6], 3:[4,5,6], 4:[5,6], 5:[6], 6:[7], 7:[8], 8:[]})
xs = [0,1,2,3,4,5]

for i in range(1,7):
    H = G.subgraph(xs[:i]+[6,7,8])
    print(third_dist_ideal(H,Equis(H.vertices(sort=True))))
\end{lstlisting}

\end{document}